\documentclass[11pt]{amsart}
\usepackage{fancyhdr}
\usepackage{txfonts}
\usepackage{amsfonts,amssymb,verbatim,amsmath,amsthm,latexsym,textcomp,amscd}
\usepackage{latexsym,amsfonts,amssymb,epsfig,verbatim}
\usepackage{amsmath,amsthm,amssymb,latexsym,graphics,textcomp, hyperref}
\usepackage{graphicx}
\usepackage{color}
\usepackage{url}
\usepackage{psfrag}
\usepackage{amsmath,amsthm,amssymb,hyperref, mdframed}
\usepackage{tikz}
\usepackage[all]{xy}
\usetikzlibrary{decorations.markings}
\usetikzlibrary{backgrounds,shapes}
\usepackage{subfig}

\usepackage{color}

\definecolor{aliceblue}{rgb}{0.9, 0.95, 1.0}

\newcommand\Z{{\mathbb Z}}

\newcommand{\C}{{\mathbb C}}

\newcommand{\pslr}{{\mathrm{PSL}_2 (\mathbb{R})}}

\newcommand{\pslc}{{\mathrm{PSL}_2 (\mathbb{C})}}

\newcommand*{\bigchi}{\mbox{\large$\chi$}}

\newtheorem{theorem}{Theorem}[section]
\newtheorem{prop}[theorem]{Proposition}

\newtheorem{cor}[theorem]{Corollary}

\newtheorem{thm}{Theorem}[section]

\newtheorem{lem}[theorem]{Lemma}
\newtheorem{defn}[theorem]{Definition}

\newcommand\CC{{\mathcal C}}
\newcommand\DD{{\mathcal D}}

\newcommand\FF{{\mathcal F}}
\newcommand\GG{{\mathcal G}}

\newcommand\LL{{\mathcal L}}
\newcommand\MM{{\mathcal M}}

\newcommand\PP{{\mathcal P}}

\newcommand\TT{{\mathcal T}}

\newcommand{\cp}{\mathbb{C}\mathrm{P}^1}

\newcommand\PMF{{\PP\kern-2pt\MM\FF}}

\newcommand\PML{{\PP\kern-2pt\MM\LL}}

\newcommand\D{{\mathbb D}}
\newcommand\R{{\mathbb R}}

\newcommand\til{\widetilde}

\title[Meromorphic projective structures, grafting and the monodromy map]{Meromorphic projective structures,\\grafting and the monodromy map}

\author{Subhojoy Gupta}

\address{Department of Mathematics, Indian Institute of Science,
Bangalore 560012, India}
\email{subhojoy@iisc.ac.in}

\author{Mahan Mj}

\address{School of Mathematics, Tata Institute of Fundamental Research, Homi 
Bhabha Road, Mumbai 400005, India}
\email{mahan@math.tifr.res.in}

\begin{document}

\begin{abstract} A meromorphic projective structure on a punctured Riemann surface $X\setminus P$  is determined, after fixing a standard projective structure on $X$,  by a meromorphic  quadratic differential with poles of order three or more at each puncture in $P$.  In this article we prove the analogue of Thurston's grafting theorem for such meromorphic projective structures, that involves grafting crowned hyperbolic surfaces.  
This also provides a grafting description for projective structures on $\mathbb{C}$ that have polynomial Schwarzian derivatives.  As an application of our main result, we prove the analogue of a result of Hejhal,  namely,  we show that the monodromy map to the decorated character variety (in the sense of Fock-Goncharov) is a local homeomorphism.  
\end{abstract}

\maketitle

\tableofcontents
\section{Introduction}

 Let $S$ be a closed oriented surface of genus $g\geq 2$. A \textit{marked complex projective structure} on $S$ is a geometric structure modeled on $\cp$, that is, it comprises an atlas of charts to $\cp$ with transition maps that are restrictions of elements of $\pslc$. Passing to the universal cover, this yields a \textit{developing map} $f:\widetilde{S} \to \cp$ that is $\rho$-equivariant where $\rho:\pi_1(S) \to \pslc$ is the \textit{holonomy} of the projective structure.

 Complex-analytically, a projective structure on $S$  is obtained by fixing a reference projective structure, and solving the Schwarzian equation
 \begin{equation}\label{schw}
 u^{\prime\prime} + \frac{1}{2} q u = 0 
 \end{equation}
on $\tilde{S}$, where $q$ is the lift of a quadratic differential on $S$ that is holomorphic with respect to a choice of complex structure.  In particular, the developing map is obtained as the ratio of a pair of linearly independent solutions, and the holonomy homomorphism records the monodromy of the solutions around homotopically non-trivial loops on the surface. 

Conversely, given a projective structure, the Schwarzian derivative of the developing map yields a quadratic differential  on $\widetilde{S} \cong \mathbb{D}$ that is invariant under the Fuchsian group  $\Gamma$ determined by the choice of complex structure on $S$; this gives back the holomorphic quadratic differential $q$ on the quotient surface. (See Proposition \ref{projq}.) 

The space of marked projective structures $\mathcal{P}_g$  then forms a bundle over  Teichm\"{u}ller space $\mathcal{T}_g$ that is affine with respect to the vector bundle $\mathcal{Q}_g$ of quadratic differentials. 

A more geometric description of a projective structure was provided by Thurston, who showed that one can obtain projective structures by starting with a hyperbolic surface (a \textit{Fuchsian} projective structure), and grafting along a measured geodesic lamination. Indeed, the resulting \textit{grafting map}
\begin{equation}\label{graft}
Gr: \mathcal{T}_g \times \mathcal{ML} \to \mathcal{P}_g
\end{equation}
is then a homeomorphism. See \S2.4 for references, and a sketch of the proof.

Recently, Allegretti and Bridgeland  \cite{AllBrid} introduced the  space of \textit{meromorphic} projective structures where the quadratic differential (in Equation  \eqref{schw}) is allowed to have higher order poles. Such meromorphic projective structures can be thought of as arising from certain degenerations of projective structures in $\mathcal{P}_g$; indeed, meromorphic quadratic differentials naturally arise in a compactification of the bundle $\mathcal{Q}_g$ (see for example \cite{BCGGM2}).  Our aim in this article is to extend Thurston's geometric description to include such structures, and also provide parametrizations of the corresponding new spaces that we need to define (see Equation \eqref{graft2}).

If there are $k\geq 1$ poles of orders given by the $k$-tuple $\mathfrak{n} = (n_1,n_2, \ldots , n_k)$ where each $n_i$ is an integer greater than or equal to 3, then we denote the corresponding space of marked meromorphic projective structures by $\mathcal{P}_g(\mathfrak{n})$. 
Here, the marking records a real ``twist" parameter at each pole (see \S3.1 for details).

The replacement of Fuchsian structures on closed surfaces (in Thurston's description) are hyperbolic surfaces with ``crown ends",  where each crown end comprises a collection of bi-infinite geodesics enclosing \textit{boundary cusps}. For any fixed tuple of integers   $\mathfrak{n}$  as above, let $\TT_g(\mathfrak{n})$ be the space of marked hyperbolic surfaces of genus $g$ and $k$ crowns, with their respective numbers  of boundary cusps given by $(n_i-2)$ for $1\leq i\leq k$. Once again, the marking not only provides a labeling of the crown ends, and the boundary cusps of each, but also ``twist" data for each crown end. It can be shown that  $\TT_g(\mathfrak{n}) \cong \mathbb{R}^\chi$ where $\chi =6g-6 + \sum\limits_{i=1}^k (n_i+1)$ (see \cite{GupWild}). 

A measured lamination on a crowned hyperbolic surface could have weighted geodesic arcs going out towards a boundary cusp, in addition to  components that are compactly supported, and we shall always include the geodesic sides of each crown end, each of  infinite weight. The space of such measured laminations ${ML}_g(\mathfrak{n})$ is also homeomorphic to $\mathbb{R}^\chi$  -- see Theorem \ref{mln} in \S3.4, which relies on a combinatorial argument that we defer to the Appendix.

Our main result is:

\begin{thm}[Meromorphic grafting theorem]\label{thm1} 
Fix integers $g\geq 0$, $k\geq 1$ such that $2g+k>2$ and a $k$-tuple  $\mathfrak{n} = (n_1,n_2, \ldots , n_k)$ where each $n_i \geq 3$. Any meromorphic projective structure $P \in \mathcal{P}_g(\mathfrak{n})$ can be obtained by starting with a crowned hyperbolic surface $\hat{X} \in  \TT_g(\mathfrak{n})$, and grafting along a measured geodesic lamination $\lambda \in {ML}_g(\mathfrak{n})$. 

This construction is uniquely determined by the projective surface $P$. Moreover,  the grafting map
\begin{equation}\label{graft2}
\widehat{Gr}: \TT_g(\mathfrak{n})  \times {ML}_g(\mathfrak{n}) \to  \mathcal{P}_g(\mathfrak{n})
\end{equation}
is a homeomorphism.
\end{thm}

\noindent\textit{Remark.}  From the definitions of the spaces (see \S3, and the preceding discussion) together with Lemma \ref{dimen} and Theorem \ref{mln} it shall follow that the two sides are indeed homeomorphic to cells of the same dimension. \\


\smallskip


Note that Thurston's  construction of the \textit{inverse} map to  Equation \eqref{graft} can be carried out for the equivariant projective structure $\widetilde{P}$  on the universal cover of the surface; we give details of the procedure in \S2.5, following \cite{KamTan}, \cite{Tanig}, and \cite{Kulkarni-Pinkall}.  In particular, this yields \textit{some} measured lamination on the Poincar\'{e} disk, invariant under \textit{some} Fuchsian group, grafting along which yields $\widetilde{P}$ (see Theorem \ref{thu-con}). Theorem \ref{thm1} precisely determines the geometry of the hyperbolic surface and measured laminations we obtain in the quotient, when we start with a {meromorphic} projective structure in the space $ \mathcal{P}_g(\mathfrak{n})$. The proof in \S4 shall crucially depend on the asymptotics of the developing map in the neighborhood of the poles, culled from classical work in the theory of linear differential systems.\\

In \S5,  we recall work of Sibuya (\cite{Sib-book}) concerning solutions to the Schwarzian equation for polynomial quadratic differentials on the complex plane.
  The proof of Theorem 1.1 also applies to this setting, and yields  the following description of the space of the corresponding projective structures on $\C$,  which could be of independent interest (see \S5.1):

 \begin{thm}\label{thm2} 
 	 For $d\geq 2$, let  $\mathcal{P}(d)$ be the space of meromorphic projective structures on $\C$ that correspond to  polynomial quadratic differentials of degree $d$.
 	Then there is a grafting parametrization
 		\begin{equation}\label{graft3}
 		\widehat{Gr}_{\C}: \text{Poly}(d)  \times \text{Diag}(d) \to  \mathcal{P}(d)
 		\end{equation}
 		where
 	\begin{itemize}
 		\item $\text{Poly}(d)$ is the space of hyperbolic ideal polygons with $(d+2)$ vertices, and 
 		\item $\text{Diag}(d)$ is the space of weighted diagonals on an ideal polygon with $(d+2)$ vertices, together with the geodesic sides of the polygon, each with infinite weight.
 	\end{itemize}

 \end{thm}

 \medskip
 
 In \S5, we provide more detailed definitions of the spaces appearing in the above theorem.  It was known from the work of Sibuya and others (see Corollary \ref{cor-dev}) that the developing maps above will have $(d+2)$ asymptotic values, where $d$ is the degree of the polynomial. Moreover, Sibuya had observed that the corresponding \textit{crown-tip map} $\Psi$  from $\mathcal{P}(d)$ to the appropriate space of $(d+2)$-tuples of points in $\cp$ (see  Equation \eqref{psimap}) is not injective.
 As an application of Theorem \ref{thm2}, we provide a characterization of the fibers of $\Psi$, that is, the set of projective structures in $\mathcal{P}(d)$ that determine the same ordered tuple of asymptotic values (called `crown tips') -- see Theorem \ref{nonuniq} for the complete statement.  \\

 For closed surfaces, the grafting description for projective structures has been useful in the study of the monodromy (or holonomy) map 
 \begin{equation}\label{holm} 
 \text{hol}:  \mathcal{P}_g \to \bigchi_{g}
 \end{equation}
 from $\mathcal{P}_g$ to the $\pslc$-character variety of surface-group representations  (see, for example, \cite{Baba2} and \cite{BabaGup}).

Here, we define a monodromy map $\Phi$ (see  \eqref{mmap}) from the space of meromorphic projective structures $\mathcal{P}_g(\mathfrak{n})$ to the \textit{decorated character variety} ${\widehat{\bigchi}}_{g,k}(\mathfrak{n})$ that records, in addition to the $\pslc$-representation of the punctured surface, the additional data of the crown-tips at each pole. See \S6.1 for a definition, that follows that of the \textit{moduli stack of framed local systems}  of Fock-Goncharov in \cite{FG} (see also \S4 of \cite{AllBrid}). \\

As an application of our main result, Theorem \ref{thm1}, we shall prove (see \S6.2):

\begin{thm}\label{thm3} The monodromy map  $\Phi: \mathcal{P}_g(\mathfrak{n}) \to {\widehat{\bigchi}}_{g,k}(\mathfrak{n})$  is a local homeomorphism.
\end{thm}

Note that it was shown in \cite{AllBrid} that this monodromy map is holomorphic, with respect to natural complex structures that these spaces acquire. Theorem \ref{thm3} thus implies that in fact $\Phi$ is a local biholomorphism. This proves the analogue of Hejhal's result for $\mathcal{P}_g$ (see \cite{Hejhal}, \cite{Earle}, \cite{Hubbard}) and confirms a conjecture of \cite{AllBrid} in our setting, where the order of each pole is greater than two.
Bakken's work in \cite{Bakk} proves a special case of Theorem \ref{thm3}, namely when $g=0$ and $k=1$, i.e.\  there is exactly one higher order pole.   The case when the order of each pole is \textit{at most} two was handled in \cite{Luo}.  \\

 Theorem \ref{thm3}  can  be thought of as an extension of the Ehresmann-Thurston principle, to our non-compact setting.  Indeed, our proof in \S6.2 shall use this principle in the usual context of compact manifolds, possibly with boundary  (see Theorem I.1.7.1 of \cite{CEG}). To be more specific, we shall apply this principle to  projective structures on the surface-with-boundary obtained by removing the crowns.  For the crown ends,  we shall exploit the fact that there are only finitely many leaves of the measured lamination entering them, that can be completed to a triangulation of the crowned surface.
 This shall allow us  to use a theorem of Fock-Goncharov (Theorem 1.1 of \cite{FG}) which implies, in our setting, that the weights on these leaves are uniquely determined by the decorated monodromy. \\

 In the case of a closed surface, the  image of the monodromy map (see Equation \ref{holm}) was characterized in \cite{GKM}. Their work can be thought of as the solution of the \textit{Riemann-Hilbert problem} for the Schwarzian equation on a closed Riemann surface.
In a sequel we shall address the analogous problem for punctured surfaces, where meromorphic projective structures will play a role.

\bigskip

\textbf{Acknowledgments.} SG thanks Kingshook Biswas, Shinpei Baba and Dylan Allegretti for illuminating conversations, and acknowledges the SERB, DST (Grant no. MT/2017/000706) and the Infosys Foundation for their support.  SG also thanks TIFR Mumbai for its hospitality; this project started at the Complex Analytic Geometry discussion meeting held there in 2018. 
We also thank the International Centre for Theoretical Sciences (ICTS) for their support and organizing the program on Surface group representations and Projective Structures (Code: ICTS/sgps/2018/12). MM is supported in part by  the Department of Atomic Energy, Government of India, under project no.12-R\&D-TFR-5.01-0500. MM is  also supported in part
by an endowment of the Infosys Foundation,
a DST JC Bose Fellowship, Matrics research project grant  MTR/2017/000005, CEFIPRA  project No. 5801-1 and by grant 346300 for IMPAN from the Simons Foundation and the matching 2015-2019 Polish MNiSW fund (code: BCSim-2019-s11). \\

\section{Background}

We recall basic facts on projective structures, and of the Thurston parametrization, that will play a crucial role in the rest of the paper. 
Throughout this section, $S_g$ would be a closed oriented surface of genus $g\geq 2$, whereas $S$ will denote a closed oriented surface, with possibly finitely many punctures. 

\subsection{Projective structures}
As mentioned in \S1,  a marked projective structure on $S$ is a maximal atlas of charts to $\cp$ such that the transition maps are restrictions of M\"{o}bius transformations.  We had also mentioned that an equivalent  definition is obtained by passing to the universal cover of the surface $\widetilde{S}$, where the local charts can be patched together to define a globally defined \textit{developing map}.
Thus, a  (marked) projective structure on $S$ consists of two pieces of data:
\begin{enumerate}
\item  a developing map $f:\widetilde{S} \to \cp$, and
\item a holonomy (or monodromy) homomorphism $\rho:\pi_1(S) \to \pslc$,
\end{enumerate} 
 such that $f$ is $\rho$-equivariant, with respect to the action of $\pi_1(S)$ by deck-transformations on the universal cover $\widetilde{S}$, and the action of the M\"{o}bius group $\rho(\pi_1(S))$  on $\cp$.

Two projective structures $(f,\rho)$ and $(g,\sigma)$ are said to be equivalent if the representations $\rho$ and $\sigma$ are   conjugate by some element  $A\in \pslc$, and the pair of maps $A\circ f, \, g$ are equivariantly homotopic to each other.

For a closed surface $S_g$, the space of equivalence classes  is then the space of marked projective structures, denoted $\PP_g$.

Since a projective structure on $S_g$ automatically also defines a complex structure on the underlying surface, there is a forgetful map $\pi:\PP_g \to \TT_g$, where $\TT_g$ is the Teichm\"{u}ller space of $S_g$.

An example of a projective structure is a \textit{Fuchsian} structure, where the developing map is injective with  image a hemisphere of $\cp$ (that can be identified with $\mathbb{D}$) and the holonomy representation $\rho$ is discrete, faithful with image in $\pslr$.  
Since any Riemann surface has such a uniformizing Fuchsian structure the fibers of the above projection map $\pi$ are never empty. 
In fact, it is well-known that the fibers are parametrized by holomorphic quadratic differentials (see, for example,  \S2 of \cite{Hubbard}):

\begin{prop}\label{projq} Let $X$ be a compact Riemann surface of genus $g\geq 2$. The space of marked projective structures on $X$ forms an affine space for the vector space $Q(X)$ of holomorphic quadratic differentials on $X$.
\end{prop}
\begin{proof}[Proof sketch.] 
The difference of two projective structures $C_1$ and $C_2$ is given by a holomorphic quadratic differential $q$, namely if $f:U \to \cp$ is the transition map between the two structures, then
\begin{equation}\label{schwd}
q = \left( \frac{f^{\prime\prime}}{f^\prime}\right)^\prime - \frac{1}{2} \left( \frac{f^{\prime\prime}}{f^\prime}\right)^2
\end{equation}
where the right hand side is the \textit{Schwarzian derivative} of $f$. 

Conversely, it is not hard to check that if $u_1$ and $u_2$ are  two linearly independent solutions of Equation \eqref{schw}, then the ratio $f:=u_1/u_2$ has Schwarzian derivative $q$.  Then, given a projective structure $C_1$, with developing map $f_1$, the new projective structure $C_2$ has a developing map given by $f\circ f_1$. 
\end{proof}

By the Riemann-Roch theorem, we know the dimension of $Q(X)$, and we immediately obtain: 

\begin{cor} The space $\PP_g$  of marked projective structures on $S_g$  is homeomorphic to $\mathbb{R} ^{12g-12}$.
\end{cor}

\textit{Remark.} In fact, $\PP_g$ is a complex manifold of dimension $6g-6$; see \cite{Hubbard}.

\subsection{Grafting}

Let $(f,\rho)$ be a Fuchsian projective structure $P_0$  on $S$. In what follows $\Gamma < \pslr$ shall be the Fuchsian group realized as the image of the holonomy map $\rho$. Note that the image of the developing map can be taken to be the upper hemisphere $U$ of $\cp$.  The operation of grafting deforms this to a different projective structure, as we shall now describe. 

Fix a  simple closed curve $a \in \pi_1(S)$, let $g : = \rho(a)$. 
Let $\alpha$ be an arc in $U$ preserved by the infinite-cyclic subgroup of $\pslr$  generated by $g\in \Gamma$.  Let $\Gamma \cdot \alpha$ be the collection of arcs stabilized by  conjugates of $\langle g\rangle$ under $\Gamma$. 

Then for a positive real parameter $t>0$, a $t$-grafting of $P_0$ along $\alpha$ is obtained by rotating one side of  each arc in $\Gamma \cdot \alpha$ relative to the other, by angle equal to $t$. A \textit{lune} is the resulting region between $\alpha$ and its rotated copy.  (See Figure 1.)

 Let $\Omega$ be the new domain on $\cp$ obtained from $U$ by the insertion of this $\Gamma$-invariant collection of lunes, one of angle $t$ at each translate $\gamma \cdot  \alpha$ where $\gamma \in \Gamma$. Then $\Omega$ is the developing image of a new projective structure $P$ on $S$.

Recall that   $\cp$ can be thought of as the boundary at infinity of hyperbolic $3$-space $\mathbb{H}^3$.  For any domain on $\cp$ invariant under a M\"{o}bius group, there is an invariant geometric object in the interior of $\mathbb{H}^3$, namely the boundary of the geodesic convex hull  (see \cite{thurstonnotes}, and our later discussion in \S2.4).  In particular, the boundary of the convex hull of the upper hemisphere $U$ is the equatorial plane, and that of the new domain $\Omega$ is a ``pleated" plane  that is bent along the geodesic axis $\gamma_\alpha$ joining the endpoints of $\alpha$, and its $\Gamma$-translates.  Here, a ``bending" is a relative rotation of one side of the geodesic axis $\gamma_\alpha$  by angle $t$  that corresponds to an elliptic element $E_{t\alpha}$ in $\pslc$. 

The deformation of $\rho$ to a new holonomy homomorphism $\rho^\prime:\pi_1(S) \to \pslc$ is best described in terms of a ``bending cocycle". We sketch the construction below -- for details, see  \S5.3 of \cite{Dum}, or II.3.5 of \cite{Epstein-Marden}.

To start, we ``straighten" the arcs $\Gamma \cdot \alpha$ to their geodesic representatives, namely consider the collection $\tilde{\gamma}$ of the geodesic axes of the hyperbolic element $g$ and its conjugates. 
The bending cocycle is then a map 
\begin{equation}\label{bcoc} 
\beta: \mathbb{H}^2 \setminus \tilde{\gamma} \times \mathbb{H}^2 \setminus \tilde{\gamma} \to \pslc
\end{equation}
where $\beta(x,y)$ defined as follows: consider the oriented geodesic arc $\sigma$ from $x$ to $y$, and let $g_1,g_2,\ldots g_n$ be the geodesics from $\tilde{\gamma}$ that intersect $\sigma$, in that order, each oriented so that $y$ lies to its right. Then $\beta(x,y) := E_1\circ E_2 \circ \cdots E_n$ where $E_i$ is the elliptic element that fixes the axis $g_i$  and rotates clockwise by an angle equal to $t$.  Note that if $\sigma \cap \tilde{\gamma} = \emptyset$, then we set $\beta(x,y) := Id$. 
If we fix a basepoint $x_0 \in \mathbb{H}^2 \setminus \tilde{\gamma}$, then the new representation $\rho^\prime$ is defined by:
\begin{equation}\label{rhop} 
\rho^\prime(c) = \beta(x_0 , c\cdot x_0) \circ \rho(c) 
\end{equation}
for any $c\in \pi_1(S)$.
Indeed, the domain $\Omega$ is invariant under the new M\"{o}bius group $\Gamma^\prime =\rho^\prime(\pi_1(S))$; the element $\rho^\prime(\gamma)$ (resp. its conjugates), acts by translations along the lune inserted at $\alpha$ (resp. its $\Gamma$-translates) and the new projective surface $\Omega/\Gamma^\prime$ is obtained by grafting a projective annulus  at $\gamma$ on the original hyperbolic surface $U/\Gamma$.  

\begin{figure}
	
  \centering
  \includegraphics[scale=0.4]{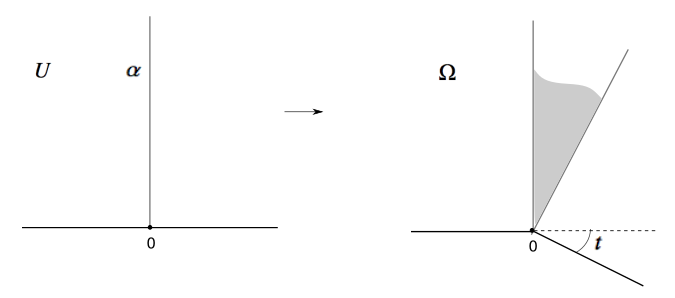}\\
  
  \caption{Grafting in a lune of angle $t$ at an arc $\alpha$.}
  \label{lune}
\end{figure}

\subsubsection*{Straight lunes}
 In the grafting construction the resulting projective structures are isotopic if the grafting arc $\alpha$ is changed by an isotopy; in particular, they remain unchanged in $\PP_g$. In particular, any lune can be isotoped to a \textit{straight lune} which is bounded by circular arcs in $\cp$, for example one obtained by grafting along a geodesic line $\alpha$.

\subsection{Measured laminations} 

Given a hyperbolic structure on $S$, a  \textit{geodesic lamination} is a closed subset that is foliated by disjoint, complete geodesics.
A collection of disjoint simple closed geodesics is certainly an example, but a geodesic  lamination could also have dense leaves, that is infinite geodesics which accumulate on to the entire lamination. (A lamination, all whose leaves are dense, is also called \textit{minimal}.)  
A geodesic lamination  is \textit{measured} if it is equipped with a transverse measure, that is, a positive measure on arcs transverse to the leaves, that is invariant under transverse homotopy.

Such a measured lamination can in fact be recovered from transverse measures of finitely many closed curves (which are also called their ``intersection numbers"). A measured lamination is thus a topological object that can be defined independent of a hyperbolic metric, as long as the surface has a marking.  The space $\mathcal{ML}_g$ of such measured laminations on $S_g$  is homeomorphic to $\mathbb{R}^{6g-6}$ (see \cite{FLP}), where the topology is induced by the transverse measures.

Note that if the hyperbolic structure is given by a Fuchsian group $\Gamma$, a geodesic lamination determines a closed set  $F \subset \GG :=\partial U \times \partial U \setminus \Delta$, where $U$ is the upper hemisphere of $\cp$, identified with $\mathbb{H}^2$, and $\Delta$ is the diagonal, and a transverse measure is a measure supported on this subset.  Any such measured lamination is then a limit of a sequence of weighted multicurves, which correspond to finite sums of Dirac measures converging in the weak-$\ast$ topology.

We can then define grafting of the Fuchsian structure along a measured lamination: a new domain $\Omega \subset \cp$ is obtained as a limit of the construction described in \S2.2,  where at each stage  we insert lunes corresponding to the weighted geodesics in the finite approximation of the lamination, mentioned above.  A similar limiting construction defines the bending cocycle Equation \eqref{bcoc} that determines the new holonomy representation $\rho^\prime$ exactly as in Equation \eqref{rhop}.  

Together, these define a new projective structure.

\subsection{Thurston parametrization} 
In the previous subsections, we have discussed how a Fuchsian structure $X$ can be grafted along a measured geodesic lamination $\lambda$ to define a new complex projective surface.  As mentioned in the Introduction, Thurston showed that this provides a unique construction of \textit{any} projective structure on a closed surface $S_g$ (see Equation \eqref{graft}). 

In this section, we discuss a statement that can be culled from the work of Kulkarni-Pinkall (\textit{c.f.} Theorem 10.6 of \cite{Kulkarni-Pinkall}) and Kamishima-Tan (\cite{KamTan}).  For a recent exposition, see \cite{BabaExp}. 
 As usual, a Riemann surface equipped with a complex projective structure will be called a  projective surface.
\begin{defn}
	A \text{maximal disk} on a projective surface $\tilde{X}$ is an embedded disk $U$ such that the restriction of $f$ to $U$ is a diffeomorphism onto a round disk  $f(U)$ in $\cp$; and  $U$ is not strictly contained in another disk with the same property.
\end{defn}

\begin{thm}\label{thu-con} 
Let $\tilde{X}$ be a simply-connected projective surface   that is not projectively isomorphic to $\C$, or the universal cover of $\cp \setminus \{0,\infty\}$. Then there exists a unique measured lamination $L$ on the Poincar\'{e} disk $\mathbb{D}$ such that  $\tilde{X}$ is obtained by grafting $\mathbb{D}$ along $L$. 

The map associating $L$ to $\tilde{X}$ is equivariant, i.e.\ if $\tilde{X}$ is the universal cover of a projective surface $S$, and the developing map $\tilde{X} \to \C P^1$ is $\pi_1(S)-$equivariant via
a representation  $\rho_\C: \pi_1(S)\to \pslc$, then $L$ is invariant under a naturally associated representation $\rho_\R: \pi_1(S)\to\pslr$. Moreover, the image $\Gamma$ of $\rho_\R$ is discrete, and the  quotient $\D/\Gamma$ is homeomorphic to $S$. 

Finally, the map $\tilde{X} \to L$ is continuous.
\end{thm}

\begin{proof}[Sketch of the proof] 
We follow the exposition in \cite{KamTan} with some differences in terminology arising from the fact that their work concerns conformally flat structures on manifolds of possibly higher dimension, of which projective structures on surfaces is a special case.

The goal is to construct a  pleated surface \cite[Chapter 8]{thurstonnotes} canonically associated to $\tilde{X}$.
Let $f:\tilde{X} \to \cp$ be the developing map of the projective structure. 
 Since $\tilde{X}$ is not projectively equivalent to the standard structure on $\C$, it follows that each point of $\tilde{X}$ is contained in a proper maximal disk (see Proposition 1.1.3 of \cite{KamTan}). 

A maximal disk $U$ acquires a natural Poincar\'{e} metric; define the set $U_\infty$ to be the subset of $\partial_\infty U$ that does not lie in $\tilde{X}$, and let $C(U_\infty)$ denote its projective convex hull in $U$.  Maximality guarantees that there are at least two points in $U_\infty$, so that the convex hull is non-empty.  Moreover,   each point of $\tilde{X}$ lies in the projective convex hull of a unique maximal disk \cite[Theorem 1.2.7]{KamTan}. 

Note that the image of a maximal disk $U$ under the developing map is a round disk $f(U)$ on $\cp = \partial \mathbb{H}^3$.  The disk $f(U)$ admits a canonical projection $\Phi_U$ to a totally geodesic copy of $\mathbb{H}^2 \subset \mathbb{H}^3$. Thus, $\Phi_U (U)$ is the convex hull of $\partial_\infty U$ in $\mathbb{H}^3$. Note that $\Phi_U(C(U_\infty))$ is an ideal totally geodesic hyperbolic polygon contained in $\Phi_U (U)$.  We have assume in our hypotheses in the Theorem that the projective surface $\tilde{X}$ is not the universal cover of $\cp \setminus \{0,\infty\}$; this guarantees that there exists at least one such polygon $\Phi_U(C(U_\infty))$ that is not degenerate, i.e.\ has at least three sides. The pleated surface below is constructed  from the collection of $\Phi_U(C(U_\infty))$'s as
follows.

Define a map $\Psi: \tilde{X} \to \mathbb{H}^3$ by $\Psi(x) = \Phi_U(f(x))$ if $x\in C(U_\infty)$. 
It is easy to verify that $\Psi$ is continuous,  and the image of $\Psi$ is a pleated plane $\PP$, in the sense of Thurston \cite[Chapter 8]{thurstonnotes}. 
Note that $\Psi$ may not even be locally injective; indeed, a ``straight lune" in $\tilde{X}$ (see \S2.2.)  arises when a family of maximal disks which have a pair of common ideal boundary points collapses to a single geodesic line $\gamma$, giving a bi-infinite geodesic in the pleating locus \cite[Chapter 8]{thurstonnotes}. If $\gamma$ is isolated in $\LL$, then the ideal polygons or {\it plaques} on either side of $\gamma$  lie on a pair of totally geodesic half-planes that can be thought of as being obtained from a (larger) totally geodesic polygon in  $\mathbb{H}^2$ after bending along $\gamma$ by a positive angle. It is possible that
$\gamma$ is not an isolated geodesic in the pleating locus $\LL$, in which case the angle of bending  is defined as a transverse measure on  the pleating locus. The transverse measure is called the {\it bending measure} and is denoted as $\mu$.

 Straightening the pleated plane $\PP$ determines a hyperbolic plane $\mathbb{H}^2$ (or the Poincar\'{e} disk $\mathbb{D}$).  The pleating locus gives a geodesic lamination $\LL$ on $\mathbb{D}$. The lamination $\LL$  equipped with the transverse measure $\mu$ gives a measured lamination $L$. This proves the first statement of the Theorem.
 
 We now observe equivariance. It suffices to show that $\Psi: \tilde{X} \to \mathbb{H}^3$ taking $\tilde{X}$ to a pleated surface is equivariant. To see this, note that for $U$ a maximal disk in $\til X$, so is $g.U$ for any $g \in \pi_1(S)$. Hence
 $$\Psi_{g.U}(C(g.U_\infty)) = \rho_\C(g)(\Psi_U(C(U_\infty))),$$ where the $\rho_\C(g)-$action on the RHS is via hyperbolic isometries.
 It follows that the totally geodesic hyperbolic polygons in $\PP$ are equivariant with respect to the action of $\rho_\C (\pi_1(S))$.
 Hence the pleating locus $\LL$, realized as a family of geodesics in $ \mathbb{H}^3$ is also equivariant with respect to the action of $\rho_\C (\pi_1(S))$.
 Next, note that the transverse measure on $\LL$ is given by the bending measure $\mu$. The latter determines and is determined by the straight lunes that occur in $\tilde{X}$. Since the developing map $f$ is equivariant under $\rho_\C$, the bending measure $\mu$ is invariant under the induced $\pi_1(S)-$action on $\PP$. Hence
the measured lamination $L$ is invariant under the induced $\pi_1(S)-$action on $\mathbb{H}^2$, where the latter is obtained from $\PP$ by straightening. Consequently, we obtain a representation  $\rho_\R: \pi_1(S)\to\pslr$, such that  $L$  is $\Gamma$-invariant, where $\Gamma$ is the image of $\rho_\R$.  Moreover, since the lunes that get collapsed by the map $\Psi$ are contractible, one can show that $\Psi$ induces a homotopy equivalence between the quotient spaces $S$ and $\D/\Gamma$. It is a standard topological fact that in this case this implies $S$ and $\D/\Gamma$ are homeomorphic; in particular the representation $\rho_\R$ is discrete and faithful. This proves the second statement of the Theorem.

Lastly, we observe the continuity of the  map $\tilde{X} \to L$.
As in the previous paragraph, it suffices to note the continuity of the map associating the projective surface  $\tilde{X}$ to the pleated plane $\PP$. This follows from the fact that the pleated plane $\PP$ depends continuously on the family   of totally geodesic polygons $\Psi_U(C(U_\infty))$, while the latter depends continuously on the family   of  projective polygons $C(U_\infty))$.
This proves the third statement of the Theorem.
\end{proof}

\textit{Remark.}
We refer the reader to \cite[Chapter 8]{thurstonnotes} for more details on pleated surfaces and to \cite[Chapter 9]{thurstonnotes}
for realizability of measured laminations via pleated surfaces. We also note that the map  $\tilde{X} \to L$ associating a measured lamination $L$ on $\mathbb{H}^2$ to a projective surface $\tilde{X} $ is exactly the inverse of the grafting map that obtains the projective surface $\tilde{X} $ from $\mathbb{H}^2$ by grafting according to the measured lamination $L$. Together with the equivariance statement of Theorem \ref{thu-con}, this proves Thurston's theorem, namely, the map $Gr$ in Equation \eqref{graft} is a homeomorphism. \\

We shall also use the following terminology:

\begin{defn}\label{grl} Given a projective structure $P$,  a \textit{grafting lamination} $L$ for $P$ on a hyperbolic surface $X$ is a measured lamination such that grafting $X$ along $L$ yields $P$.
\end{defn}

\section{Meromorphic projective structures and crowned hyperbolic surfaces}

In this section, we shall provide a more detailed exposition of some of the objects and their spaces already introduced in \S1, in particular, those appearing in  the statement of Theorem 1.1 (see the map defined by Equation \ref{graft2}). 

\subsection{Meromorphic projective structures and their markings} For a Riemann surface with punctures, \cite{AllBrid} considered projective structures obtained by solutions of Equation \eqref{schw} when $q$ is holomorphic away from the punctures, and has poles of finite order, greater than two, at the punctures.  Poles of order one already appear in classical Teichm\"{u}ller theory: for Fuchsian structures they arise when the uniformizing structure has a  finite-volume \textit{cusp} at the puncture.  Examples of projective structures  corresponding to meromorphic quadratic differentials with poles of order two include \textit{branched} structures;  see \cite{Luo}. \\

Recall from  the proof of Theorem \ref{projq} that the ``difference" of two projective structures, given by the Schwarzian derivative of the transition maps between charts in the two structures,  is a holomorphic quadratic differential.  Following the definition in \S3.3 of \cite{AllBrid}, we say: 

\begin{defn}\label{merp}  A meromorphic projective structure is a projective structure on a punctured Riemann surface $X \setminus P$ such that the difference (in the sense described above) with the restriction of a standard (holomorphic) projective structure on $X$ is given by a holomorphic quadratic differential on $X\setminus P$ that extends to a meromorphic quadratic differential $q$ with  poles of order greater than two at each $p\in P$. 

If, in a choice of a   coordinate disk  around a pole, $q$ has the expression
\begin{equation}\label{polar}
q = \left(\frac{a_n}{z^n} + \frac{a_{n-1}}{z^{n-1}}  + \cdots + \frac{a_1}{z}  + h(z)\right)dz^2
\end{equation}
where $h(z)$ is a holomorphic function, then the \textit{polar part} of the differential is defined to be $q - h(z)dz^2$.  
 \end{defn}
 
 \textit{Remarks.} 1. We shall assume the standard projective structure on $X$ is the uniformizing one, which in case the Euler characteristic $\chi(X) <0$ is  hyperbolic, if $\chi(X) =0$ is a quotient of $\C$, else is the projective surface $\cp$ itself. \\
 2. Unlike in \cite{AllBrid}, our definition above disallows poles of order two (or ``regular" singularities); this shall make our defining spaces of structures simpler, as our projective structures shall automatically have no ``apparent singularities". \\
 
 Recall that the \textit{horizontal} directions of a quadratic differential $q$ at a point are the tangent directions in which the differential takes real and positive values. In a  neighborhood of  a pole of order $n\geq 3$, as in Equation \eqref{polar}, the quadratic differential $q$ has $(n-2)$ equispaced directions at the pole that horizontal trajectories are asymptotic to (see Theorem 7.4 of \cite{Strebel}).
 
 \textit{Example.} For the quadratic differential $q= z^{-n}dz^2$ where $n\geq 3$, these horizontal directions at the pole are at  the points $\{\text{exp}(2\pi j/(n-2)) \mid 0 \leq j < n-2\}$ on the unit circle on the tangent plane obtained by a real blow-up at the pole.\\
 
 We also define:
 
 \begin{defn}\label{mark}
 A \textit{marking} of a meromorphic projective structure on $X\setminus P$ is a choice of a homeomorphism (up to homotopy) with a surface $S$ with boundary $C$, where each component of $C$ has (a positive number of) labeled marked points on it.   The homeomorphism  takes the horizontal directions at each pole to the marked points on a corresponding boundary component.  Here we consider two homeomorphisms the same if they are homotopic relative to the boundary (that is, by a homotopy that keeps the boundary fixed pointwise).
 
 As mentioned in \S1, $\mathcal{P}_g(\mathfrak{n})$ shall denote the space of marked meromorphic projective structures with $k$ poles of orders given by the tuple $\mathfrak{n} = (n_1,n_2, \ldots, n_k)$ , where each $n_i \geq 3$. We shall assume $2g+k>2$, that is, the underlying surface has negative Euler characteristic. 

 \end{defn}
 
 \textit{Remark.} Note that under the above notion of equivalence of two marked surfaces,    two markings that differ by a Dehn twist around the boundary component are distinct. \\
  
 It is useful to also consider an ``appended"  Teichm\"{u}ller space of the underlying marked Riemann surfaces :
 
 \begin{defn}\label{hatT} Let $S$ be an oriented surface of genus $g$ and $k$ punctures, having negative Euler characteristic, and  let $\mathfrak{n}$ be a $k$-tuple of integers as above. Then the space $\hat{\mathcal{T}}_{g,k}$ shall denote the space of marked complex structures on $S$, together with an additional real parameter $r_i$ at the $i$-th  puncture, for $1\leq i\leq k$. 
Note that a  marking includes a labeling of the punctures, and is considered up to a homotopy as in Definition \ref{mark}.
The real parameter $r_i$  serves to record:
\begin{itemize}
 \item[(a)] A set of $(n_i-2)$ equispaced points on a circle obtained as a real blowup of the $i$-th puncture, where the first point  is at $\text{exp}(i 2\pi r_i)$, and
 \item[(b)] The integer parameter $\lfloor r_i \rfloor$ that denotes the number of Dehn twists about a boundary circle obtained from a real blowup of the $i$-th puncture. 
 \end{itemize}
 
\end{defn}
 
 \textit{Remark.}  Recall that for the Teichm\"{u}ller space of a punctured surface, the puncture is thought of as a boundary component of length zero. The ``appended" Teichm\"{u}ller space defined above can be thought of as adjoining an extra Fenchel-Nielsen twist parameter about this boundary curve. \\
 
  Note that there is a projection  $\pi: \mathcal{P}_g(\mathfrak{n}) \to \hat{\mathcal{T}}_{g,k}$ that maps a meromorphic projective structure to the punctured Riemann surface underlying it, which at each puncture has 
  \begin{itemize}
  \item a set of equispaced points on the circle obtained as its real blowup,  given by the horizontal directions of the meromorphic quadratic differential, and 
  \item a marking that remembers the twist parameter; in particular, the number of Dehn-twists  around the corresponding boundary component.
  \end{itemize}

\medskip

 We then have:
 
 \begin{lem}\label{dimen} The space $\mathcal{P}_g(\mathfrak{n})$ is homeomorphic to $\R^{2\chi}$ where $\chi =6g-6 + \sum\limits_{i=1}^k (n_i+1)$.
 \end{lem}
 
 \begin{proof}
 Clearly $ \hat{\mathcal{T}}_{g,k} \cong \R^{6g-6+3k}$ since the usual Teichm\"{u}ller space  of a genus-$g$ surface with $k$ punctures is homeomorphic to $\R^{6g-6+2k}$ (\textit{c.f.} the remark following Definition \ref{hatT}). 
 Fix a marked Riemann surface $ X\in \hat{\mathcal{T}}_{g,k}$, and a coordinate chart $U_i$ around the $i$-th puncture (where $1\leq i\leq k$). 
 Then the fiber $\pi^{-1}(X)$ of meromorphic projective structures that project to $X$, consists of meromorphic quadratic differentials that have a  pole of order $n_i$ at the $i$-th puncture, with horizontal directions as prescribed by the corresponding real parameter $r_i$ on $X$. The horizontal directions at a pole are determined by the argument $\text{Arg}(a_n)$, where $n:=n_i$, and $a_n$ is the leading order coefficient of the polar part (Equation \eqref{polar}) as expressed in the chart $U_i$.

 This leaves the positive real number $\lvert a_n\rvert$, together with  the remaining coefficients $a_1,a_2,\ldots a_{n-1} \in \C$ of the polar part, a total of $(2n-1)$ parameters. The holomorphic quadratic differentials on a closed surface of genus $g$,  by Riemann-Roch, is a complex vector space of dimension $3g-3$. Hence the fiber  $\pi^{-1}(X)$  is   homeomorphic to a cell of (real) dimension ${6g - 6 + \sum\limits_{i=1}^k(2n_i -1)}$.
 
We conclude that the total space  $\mathcal{P}_g(\mathfrak{n})$ is homeomorphic to a cell of dimension $\left(6g - 6 + \sum\limits_{i=1}^k(2n_i -1) \right) + \left( 6g-6+ \sum\limits_{i=1}^k 3 \right) = 2\chi$. \end{proof}

  \bigskip

\textit{Remark.} In fact, $\mathcal{P}_g(\mathfrak{n})$ can be shown to be a complex manifold of dimension $\chi$ (see Proposition 8.2 of \cite{AllBrid}). 

\subsection{Crowned hyperbolic surfaces}

A \textit{hyperbolic crown} is an annulus equipped with a hyperbolic metric such that one of the boundary components is a closed geodesic (the \textit{crown boundary}), and the other comprises a finite chain of bi-infinite geodesics, each adjacent pair of which encloses a \textit{boundary cusp}.  The bi-infinite geodesics shall be called the \textit{geodesic sides} of the crown. 
A \textit{marking} on the hyperbolic crown is a labeling of the boundary cusps  together with a choice of a homotopy class of an arc from the crown boundary to a boundary cusp. The latter is an integer parameter that records the number of twists around the boundary component.

\begin{figure}[h]
  \centering
  \includegraphics[scale=0.5]{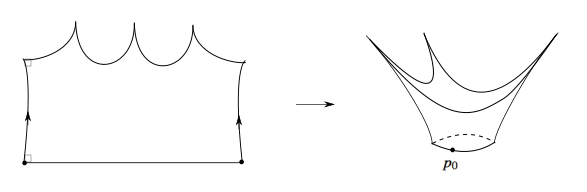}\\
  \caption{A hyperbolic crown with basepoint $p_0$ on the boundary.}
\end{figure}

A crowned hyperbolic surface $\hat{S}$ is obtained by gluing a hyperbolic crown to a hyperbolic surface with geodesic boundary $\gamma$, such that the boundary component of the crown is identified with $\gamma$. The hyperbolic crown is then a subsurface of $\hat{S}$ that we refer to as its \textit{crown end}. 

Topologically, a crowned hyperbolic surface is a surface with boundary, together with a collection of marked points on the boundary.  
A \textit{marking} on a crowned hyperbolic surface is a choice of homotopy class of an identification with such a surface, where the homotopy fixes the boundary pointwise. The latter condition amounts to fixing some boundary data (see below) that are additional parameters for specifying such a surface. 

The  ``wild" Teichm\"{u}ller space  $\TT_g(\mathfrak{n})$  introduced in \S1 (see also \cite{GupWild})  is the space of such marked crowned hyperbolic surfaces corresponding to the tuple $\mathfrak{n}$; each surface in this space has $k$ crown ends, each having $(n_i-2)$ boundary cusps.  See also \cite{Penner} for a broader context.

\subsubsection*{Boundary twist data}
A crown end of a crowned hyperbolic surface has an additional real parameter associated with it that we now describe.
Let $\gamma$ be the boundary of the crown with length $l$. Let  $\alpha$ be a fixed choice of a directed arc  between boundary cusps on the crowned hyperbolic surface $\hat{S}$, such that $\alpha$ is non-trivial in homotopy (relative to its end-points), and not peripheral in the sense that it cannot be homotoped into the crown end. (In particular, $\alpha$ intersects $\gamma$ twice.) 
First, note that the marking of the crowned surface  determines an integer twist data that records the number $t\in \Z$ of twists around $\gamma$  that  $\alpha$ makes. We shall also assume that all the twisting round $\gamma$ that $\alpha$ makes,  takes place inside the crown.

Next, a  hyperbolic crown with $m$ boundary cusps determines a  basepoint on the boundary $\gamma$: namely consider the geodesic side of the crown between the cusps labeled $m$ and $1$, and consider the foot  $p_0$ of the perpendicular that realizes the distance of that geodesic side from $\gamma$. We shall refer to this as the \textit{canonical basepoint} for the crown. 

The \textit{real}  twist parameter of the crown end is then measured relative to this canonical basepoint: let  $d$ be the distance along $\gamma$ from $p_0$ to the point  where $\alpha$  intersects $\gamma$ first (in the orientation of $\gamma$ acquired from the crown). Recall $\alpha$ completes $t$ complete twists around $\gamma$; then the twist parameter associated with the crown end is defined to be $\tau = t\cdot l + d$. 

Alternatively, instead of the choice of a directed arc $\alpha$, the twist parameter can be thought as comprising an integer twist data $t$, together with a choice of a basepoint $p$ on $\gamma$ at a distance $d \in [0,l)$ from $p_0$ on the (oriented) boundary of the crown.  As before, this can be recorded as the real number $\tau = t\cdot l + d$.\\

For the proof of the following fact,  already mentioned in the introduction, see Lemma 2.16 of \cite{GupWild}:

\begin{prop}\label{tgn} The space  $\TT_g(\mathfrak{n}) \cong \mathbb{R}^\chi$ where $\chi =6g-6 + \sum\limits_{i=1}^k (n_i+1)$. The parameters include $6g-6 + 3k$ real numbers that specify the hyperbolic surface with geodesic boundary obtained by removing the crown ends, together with parameters determining each crown, including the  boundary twist parameters as defined above. 
\end{prop}

\smallskip

\subsection{Measured laminations on crowned hyperbolic surfaces} 
As described in \S1, a measured lamination on a crowned surface could have non-compact support, with finitely many leaves that exit through the boundary cusps of the crown end. In this paper, such a lamination will also include the geodesic sides of the crown end, each of which is assigned weight $\infty$.  (See Figure 3.) 

\begin{figure}
	\centering
	\includegraphics[scale=0.5]{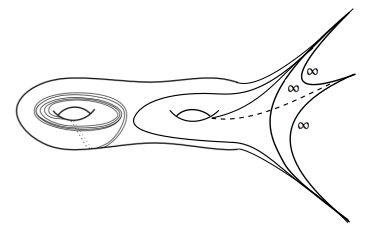}\\
	\caption{A measured lamination on a crowned hyperbolic surface.}
\end{figure}

Suppose that the crowned surface is in $\TT_g(\mathfrak{n})$.  Thus it has $k$ crown ends, where the number of boundary cusps of crown ends is given by the $k$-tuple $\mathfrak{n}$. Then the space of such measured laminations is  $\mathcal{ML}_g(\mathfrak{n})$.  Just as for $\mathcal{ML}_g$ in \S2.3, this space can be thought of as parametrizing topological objects. 

Note that the space $\mathcal{ML}_g$ of measured laminations on a \textit{closed} surface of genus $g\geq 2$ can be parametrized by weighted train-tracks (see \cite{Pen-Har}); indeed, $\mathcal{ML}_g$ acquires its topology via this parametrization. 
One way of parametrizing the space $\mathcal{ML}_g(\mathfrak{n})$ (and equipping it with a topology) would be to use weighted train-tracks \textit{with stops}, as introduced in \S1.8 of \cite{Pen-Har}. (Note that  \cite{Pen-Har} considers a single stop on each boundary component, but this can easily be extended to the case of multiple stops.)  \\

In what follows, we provide an alternate parametrization, by dividing a measured lamination $\lambda$ on a crowned hyperbolic surface into its intersections with the crown-ends, and with the surface with boundary that is the complement of the crowns.  (We shall always assume that the twisting of leaves entering a crown end around the corresponding crown boundary takes place in the crown.) 

This approach takes advantage of the fact that the parametrization of measured laminations on a surface with boundary is well-known (see, for example, Proposition 3.9 of \cite{Alps}).  In what follows we shall first  prove a similar parametrization of measured laminations on a hyperbolic crown (Proposition \ref{lamcrown}), and parametrize  $\mathcal{ML}_g(\mathfrak{n})$ by combining these two parametrizations (Proposition \ref{mln}). Part of the proof is to show that when we attach  crown ends to a surface-with-boundary, then measured laminations on the pieces can be matched up to produce a measured lamination on the crowned hyperbolic surface -- the details of this are deferred to the Appendix.

We shall implicitly assume that  $\mathcal{ML}_g(\mathfrak{n})$ acquires a topology via this parametrization. \\

We start with the following observation:

\begin{lem}\label{lamcrown0} The intersection of  the measured lamination $\lambda\in \mathcal{ML}_g(\mathfrak{n})$ with a crown end $\mathcal{C}$  is a collection of (isolated) weighted arcs, each of infinite length, that either run from a boundary cusp to the crown boundary, or between two non-adjacent boundary cusps. 
\end{lem}
\begin{proof}
In the universal cover, the boundary cusp points corresponding to a lift $\hat{P}$ of the crown $\mathcal{C}$  have precisely two accumulation points: the endpoints of the geodesic line that is the lift of the crown boundary $\gamma$ . No leaf of $\lambda$ can be asymptotic to these two points. This is because such a leaf would have to spiral infinitely many times around the closed curve $\gamma$, and therefore could not have positive transverse measure. Hence the restriction of a lift of the lamination $\lambda$ to $\hat{P}$ is a collection of geodesic lines, each having (one or both) endpoints at a set of isolated points on the ideal boundary. 

Recall that $\gamma = \partial \mathcal{C}$ is the closed geodesic that is the crown boundary. We note finally that there can be at most finitely many
geodesics in $\lambda$ that intersect $\mathcal{C}$. To see this, observe that
the intersection
$\lambda \cap \gamma$ is a closed subset of $\gamma$. For each complementary interval $I_i$ in $\gamma \setminus (\lambda \cap \gamma)$, there is an polygon $B_i$ in $\mathcal{C}$ bounded by $I_i$ on one side,  two geodesic leaves of $\lambda$ that exit the boundary cusps, and possibly some geodesic sides of the crown. If $\lambda \cap \gamma$ is infinite, there are infinitely many such distinct (and necessarily disjoint) $B_i$'s forcing the total area of the crowned hyperbolic surface to be infinite--a contradiction.
\end{proof}

In what follows, we define a \textit{measured lamination on a hyperbolic crown} to be a collection of finite weighted geodesics as above. (that we refer to as \textit{arcs}).  Note that the closed geodesic that is the crown boundary, could also be part of the lamination. We also require that there is at most one arc from a boundary cusp to the crown boundary; thus, arcs obtained by ``splitting" (see Appendix) will be considered the same arc (with a total weight equal to the sum of individual weights).

\begin{prop}\label{lamcrown} For a hyperbolic crown $\mathcal{C}$ with $(n-2)$ boundary cusps, the space of measured geodesic laminations 
on $\mathcal{C}$ is parametrized by $\R^{n-1}$. The parameters include the transverse measure $l$ of the boundary $\gamma$ of the crown, and the boundary twist parameter $\tau  \in \mathbb{R}$ , which together parametrize $\R^2$. 
\end{prop}

\begin{proof}
Recall that the bi-infinite geodesic sides of a crown end are also part of this lamination on the crown, each equipped with infinite weight. 
 A collection of disjoint weighted geodesics $\mathcal{G}$ on $\mathcal{C}$ can then be represented by a dual metric graph $G$, that we define as follows:
 
 The vertices of $G$ are one for each complementary region of $\mathcal{G}$,  and each edge of $G$ is either\\
 (a) transverse to an arc in $\mathcal{G}$ and having length equal to its weight,  and connecting the vertices in the complementary regions on either side, or \\
 (b) has infinite length, from a vertex to a geodesic side of the crown, in case the complementary region is bounded by such a side. \\
 Note that there are $(n-2)$ edges of infinite length corresponding to the $n$ geodesic sides of the crown, and if the crown boundary has positive measure, there is a unique cycle of edges corresponding to the boundary, that we shall denote by $\mathfrak{c}$.  See Figure 4.  
 
 Moreover,  the requirement of at most one arc from a boundary cusp to crown boundary, ensures that each vertex of $G$ is at least trivalent. 
 
 For any fixed positive transverse measure  $l$ of the boundary, the space of such metric graphs is  homeomorphic to $\R^{n-3}$ (see, for example, Theorem 3.3 of \ref{MulPenk}). The idea is that for a fixed topological-type of the graph, varying the lengths of the edges parametrizes a cell, and the different cells fit together to give a cell-complex that is homeomorphic to a ball. \\
  
 \begin{figure}
	\centering
	\includegraphics[scale=0.44]{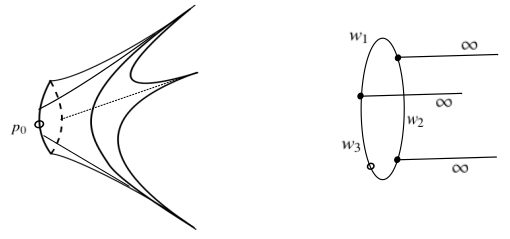}\\
	\caption{A  collection of weighted diagonals on a hyperbolic crown (left) determines a dual metric graph (right). 
	This figure depicts a case when the resulting graph is of a generic type: varying the weights parametrizes a cell of top dimension in the resulting cell-complex. }
\end{figure}

In fact, just as in \S3.4 of \cite{Alps}, one can interpret a \textit{non-positive} transverse measure $l$ the following way: in such a case, there will be no geodesic arcs incident on the crown boundary, but instead the crown boundary itself, which is a closed geodesic, will be part of the measured geodesic lamination, and will be given a weight $\lvert l \vert$.

The dual metric graph in such a case will be a tree, with $(n-2)$ edges of infinite length as before, but now with an additional finite-length edge corresponding to the closed boundary geodesic, instead of a cycle. Once again, for any fixed $l\leq 0$, the space of such dual metric trees is  homeomorphic to $\R^{n-3}$ (\textit{c.f.} Theorem 16 of \cite{GupWolf3}).\\

It remains to verify that the total space (as we vary the  transverse measure in $\R$) is also homeomorphic to a ball having \textit{two} additional dimensions;  one of these parameters is the transverse measure itself, that we denote by $l$, and the other is the boundary twist parameter $\tau$. 

However, note that the twist parameter for the crown $\mathcal{C}$ will only affect the measured lamination on it only in the case that $l>0$, for only then will there be geodesic leaves incident on the crown boundary. If $\tau = t\cdot l + d$, we shall call $t$ the \textit{integer part} of the twist parameter. This integer records the number of Dehn twists such leaves make around the crown boundary. The real part of the twist parameter $\tau$ is the real number $d \in [0,l)$, and it determines the position of the basepoint on the cycle $\mathfrak{c}$ of the dual metric graph, relative to the canonical basepoint on the crown boundary (see \S3.2). 

In the remainder of this proof we shall describe how the parameters $(l,\tau)$ still determine copy of $\R^2$. Together with the previous discussion, this would imply that the space of measured geodesic laminations on $\mathcal{C}$ is $\R^{n-3} \times \R^{2} \cong \R ^{n-1}$. 
  
 First, consider the upper half-plane $\mathsf{H} \subset \R^2$ where the height is given by the transverse measure $l$, and the twist parameter $\tau$ determines the horizontal coordinate. At a fixed height $l>0$, when the cycle $\mathfrak{c}$ has total length $l$, the range of values $0\leq \tau < l$ will correspond to $0$ integer twist, the range  $l\leq \tau < 2l$ will correspond to the integer part $t=1$, and so on. This partitions  $\mathsf{H} $ into wedge-shaped regions $V_j = \{j\cdot l \leq \tau \leq (j+1)\cdot l\}$ for $j\in \mathbb{Z}$, that represent the different integer parts of the twist parameter.

Next, we include the points where the transverse measure  $l\leq 0$  (where the twist parameter is no longer relevant) as the 
real line  boundary $\mathbb{R}$ of the upper half-plane $\mathsf{H}$  with an identification of the positive and negative half-rays; that is,  both the points $(-l,0)$ and $l,0)$ represents the same point, where the transverse measure is $l\leq 0$.

Note that the wedges $V_j$ accumulate onto these half-rays as $j\to \pm \infty$. This fits into the tiling of the interior of the upper half-plane described earlier:

If one fixes the twist parameter $\tau$, and then decreases the transverse measure of the boundary (that is, go vertically down to the boundary in $\mathsf{H}$), then the leaves of the lamination intersecting the crown boundary  will have an increasing number of twists around the boundary, but have proportionately smaller weights, and will limit (as a measured lamination) to the boundary geodesic with a weight. 

It is easy to now verify that the closed upper half-plane $\overline {\mathsf{H}}$ with the identification on the boundary half-rays as described above, is homeomorphic to $\R^2$, as we claimed. 
\end{proof}

\begin{prop}\label{mln} The space of measured laminations $\mathcal{ML}_g(\mathfrak{n})$ is homeomorphic to $\R^\chi$ where $\chi = 6g-6 + \sum\limits_{i=1}^k (n_i+1)$.
\end{prop}

\begin{proof}

 It is well-known that  the space of measured laminations on a surface $S$ of genus $g$ and $k$ boundary components is a cell of dimension $6g-6 + 3k$ (see, for example, Proposition 11 of \cite{GupWolf3}). The parameters are the transverse measure, and a twist parameter,  for each interior pants curve for a pant decomposition of the the surface-with-boundary, together with the transverse measures of the $k$ boundary components.   The case that the transverse measure of such a component is non-positive, say $l\leq 0$, can be interpreted as in the proof of Proposition \ref{lamcrown}. Namely, in that case the boundary itself is a leaf of the lamination, with weight $\lvert l \rvert$.

To the $i$-th boundary component, where $1\leq i\leq k$, we can now attach a crown end with $(n_i-2)$ boundary cusps. By Proposition \ref{lamcrown}, a measured lamination on such a crown end is determined by $(n_i-1)$ parameters, and in the Appendix we describe how such a lamination is matched with the measured lamination on the surface-with-boundary, to obtain a measured lamination on the crowned hyperbolic surface $\hat{S}$.
However, for this gluing, the transverse measures on the common boundary induced by the two laminations need to match. So, the total number of real parameters is $\chi$, as desired. 

From Lemma \ref{lamcrown0}, it is not hard to see that \textit{any} measured lamination on the crowned hyperbolic surface arises as a result of such a construction, completing the proof. 
\end{proof}

\textit{Remark.} Alternatively, such measured laminations can be shown to be equivalent to measured foliations with pole singularities,  as defined in \cite{GupWolf3} -- see, for example, \S11.8-9 of \cite{Kap} for a proof of this equivalence in the case of closed surfaces.  The latter space of measured foliations with pole singularities is parametrized in Proposition 10 of \cite{GupWolf3}, and shown to be homeomorphic to $\R^\chi$.

\subsubsection*{Grafting} The operation of grafting a crowned hyperbolic surface $\hat{S}$ along  a measured lamination  $\lambda$ on it  makes sense. As described in \S2.2 and \S2.3, we first pass to the universal cover and perform the relative bending for each of the lifts of the leaves of $\lambda$ or its finite approximations, and then take a limit. The infinite grafting for each geodesic side of the crown end (which have infinite weight) can be thought of as grafting in an infinite concatenation of lunes:  conformally this yields a half-plane. This gives us a new projective structure on the  punctured Riemann surface; we shall see later (see \S4.3) that this is in fact a meromorphic projective structure as in Definition \ref{merp}.

\section{Proof of Theorem 1.1}

In the preceding sections, we have completed defining the spaces that appear in Equation \eqref{graft2} which we reproduce below:
\begin{equation}
\widehat{Gr}: \TT_g(\mathfrak{n})  \times {ML}_g(\mathfrak{n}) \to  \mathcal{P}_g(\mathfrak{n}).
\end{equation}
Recall that $\mathfrak{n} = (n_1, n_2,\ldots, n_k)$ records the orders of the poles (each greater than two)  at the $k$ labeled punctures, and for each $1\leq i\leq k$,  we denote $m_i =(n_i-2)$ to be  the number of boundary cusps of the corresponding crown end for a surface in $\TT_g(\mathfrak{n})$. 

In this section, we complete the proof that the map $\widehat{Gr}$ is a homeomorphism. \\

For ease of notation, we shall assume throughout that $k=1$, that is, the underlying Riemann surface has a single puncture, or equivalently, the underlying hyperbolic surface has  a single crown end. Thus there is an integer $n\geq 3$  denoting the order of the pole in the Riemann surface interpretation;  equivalently, $(n-2)$ is the number of boundary cusps of the crowned hyperbolic surface. The proofs in the section only involve a local analysis around the pole, and are exactly the same for multiple punctures/crown ends. \\

\subsection{Linear differential systems and asymptotics of the solutions}
We begin with some key results from classical work on linear differential equations on the complex plane. 

Recall  that the developing map for a meromorphic projective structure is the ratio of two linearly independent solutions of the Schwarzian equation \eqref{schw}, where the quadratic differential $q$ is of the form Equation \eqref{polar} on a coordinate disk $U$ around the pole. 

Using a  change of coordinate $z\mapsto w:=c/z$ for a suitable $c$,  one can consider the (transformed) quadratic differential to be of the form 
\begin{equation}\label{qexp}
q(w) = -2 \cdot (w^{d} + \alpha_{d-1}w^{d-1} + \cdots \alpha_1w + \alpha_0 + \alpha_{-1}w^{-1} +  \cdots) 
\end{equation}
where $d=n-4$, so that it has a pole of order $n$ at $\infty$.  Our choice of the factor $(-2)$ is merely in order to match with the classical  literature (see, for example, Equations 1.1 and 1.2 of \cite{H-S}). 

The Schwarzian equation restricted to $U$ is then the equation 
\begin{equation}\label{schw2}
u^{\prime\prime}(w) + \frac{1}{2}q(w) u(w) = 0 
\end{equation}
defined on a neighborhood of $\infty$ in $\C$. 

Taking $X(w) = \left(\begin{smallmatrix}u \\u^\prime\end{smallmatrix}\right)$, the equation \eqref{schw2} can  be written as the linear system of rank two:
\begin{equation}\label{linsys1} 
X^\prime(w) = A(w) X(w) \hspace{.1in} \text{where}  \hspace{.1in} A(w) = \left(\begin{matrix}0&1\\-\frac{1}{2}q(w)&0\end{matrix}\right).
\end{equation}
\medskip

In what follows we shall assume that $d$ is even, that is, the pole is of even order. The case of an odd order pole can be reduced to this by taking a double cover branched at the pole -- see \S5.3 of \cite{AllBrid} for details. 

Note that the gauge transformation
\begin{equation*}
 X = \left(\begin{matrix}1&0\\0&w^{d/2}\end{matrix}\right)Z
\end{equation*} 
converts the system to
\begin{equation}\label{linsys2}
Z^\prime(w) = w^{d/2}\left(\displaystyle\sum\limits_{k=0}^{\infty} B_k w^{-k} \right) Z(w)
\end{equation}
with the advantage that the leading order term $B_0 = \left(\begin{matrix}0&1\\-\frac{1}{2}&0\end{matrix}\right)$ is diagonalizable.

\medskip

An analysis of this linear system can then be carried out as in the work of Hsieh-Sibuya (\cite{H-S}); though they  consider the case where $q(w)$ is a polynomial, their analysis extends to our setting where $q(w)$ is holomorphic in a neighborhood of $\infty$ with a finite order pole at $\infty$. 

In fact, this more general setting is handled for linear systems of arbitrary rank by work of Balser-Jurkat-Lutz in \cite{B-J-L} (see also Chapter XIII of \cite{H-S-book}.  The following theorem can be culled from Theorem A of \cite{B-J-L}; see also Theorem 6.1 of \cite{Sib-book}, and the exposition in \cite{Bakk} and  \S5.3 of \cite{AllBrid}. 

\begin{thm}\label{linear}  There are $(d+2)$ sectors $S_k$ in $\C$ bounded by the rays at angles $\frac{\pi}{d+2}\cdot (2k \pm1)$ where $k \in \{0,1,2,\ldots, d+1\}$, and $(d+2)$ uniquely determined fundamental solutions ${Y}_k(w)$ to Equation \eqref{schw2} that
\begin{itemize}
\item[(a)] are holomorphic in a neighborhood $U$ of $\infty$,
\item[(b)] have an asymptotic expansion 
\begin{equation}\label{asymp}
Y_k(w) = cw^\rho\left(1+ O(w^{-1/2})\right)e^{(-1)^{k+1}E(w)} 
\end{equation}
in  $S_{k-1} \cup S_k \cup S_{k+1}$,  where $c$ and $\rho$ are some constants (that may depend on $k$), and $E(w)$ is a polynomial of degree $(d/2 +1)$. 
\item[(c)] are related by 
\begin{equation}\label{yky} 
Y_k(w) = Y_0(\omega^kw)
\end{equation}
where $\omega = e^{\frac{2\pi i}{d+2}}$.

\end{itemize} 

\end{thm}

\smallskip

\textit{Remark.} The sectors are often called \textit{Stokes sectors} and the  rays between the sectors are called \textit{Stokes rays}; we shall also refer to the rays at angles $\frac{2\pi}{d+2}\cdot k$ for $k \in \{0,1,2,\ldots, d+1\}$ to be the \textit{anti-Stokes rays}, that bound the \textit{anti-Stokes sectors} that we denote by $\widehat{S}_k$.  (The latter would later feature in aspects of the corresponding projective structures:  in particular, the developing map would be asymptotic to the crown-tips along the anti-Stokes rays, and its restriction to the anti-Stokes sectors would correspond to infinite-grafting on the geodesic sides of the crown.) \\

Note that one  consequence of the asymptotics of Equation \eqref{asymp} is that  for each $k \in \{0,1,2,\ldots, d+1\}$, the fundamental solution $Y_k(w) \to 0$ as $w\to \infty$ in $S_k$, whereas $Y_{k}(w) \to \infty$ as $w\to \infty$ in $S_{k\pm1}$. The solution $Y_k$ is said to be \textit{subdominant} in the Stokes sector $S_k$.\\

In particular, this shows that $Y_0(w)$ and $Y_{1}(w)$ are linearly independent solutions of Equation \eqref{schw2}, so we can consider the developing map for the corresponding projective structure it defines to be the ratio
\begin{equation}\label{devY}
f(w) = \frac{Y_0(w)}{Y_1(w)}.
\end{equation}

\medskip

In what follows, recall that an \textit{asymptotic value} of a meromorphic function as above, defined in a neighborhood of $\infty \in \C$, is the limiting value in $\cp$ (if it exists) of the function along a curve diverging to $\infty$.

As an immediate corollary of Theorem \ref{linear} and Equation \eqref{devY} we then obtain:

\begin{cor}\label{cor-dev} The developing map $f:U \to \cp$ as defined above is holomorphic in a neighborhood of $\infty$, and has $(d+2)$ asymptotic values $c_0, c_1,\ldots , c_{d+1} \in  \cp$, one in each sector $S_k$ , for $0\leq k \leq d+1$.

Moreover, in each anti-Stokes sector, there is the asymptotic expansion:
\begin{equation}\label{fexp} 
f(\xi) \sim e^{-2\xi}
\end{equation}
in the coordinate $\xi= w^{(d+2)/2}$.

(Here, Equation \eqref{fexp} means that  $\xi^m\left(f(\xi) - e^{-2\xi}\right) \to 0$  as $\xi \to \infty$, for each $m\geq 0$.)
\end{cor} 

\begin{proof}

The asymptotics in Equation \eqref{fexp} is an immediate consequence of Equations \eqref{asymp}, \eqref{yky} and \eqref{devY}; note that the growth in a sector is dominated by  $\text{exp}(E(w))$ and the leading order term of $E(w)$ is $w^{(d+2)/2}$, so the change of coordinate is $w \mapsto \xi= w^{(d+2)/2}$. In particular, along the {anti-}Stokes rays, we have $\text{Im}(w^{(d+2)/2}) =0$,  and $\text{Re}(w^{(d+2)/2}) \to \infty$. 

We remark that Equation \eqref{fexp} can also  be derived from \S5.6 (Theorem 5.6.1) of \cite{Hille} -- see also Theorem 3.2 of \cite{Arias}.
There, the coordinate $\xi$ is described to be the natural coordinate for the quadratic differential $q$ in Equation \eqref{qexp}, that is, $\xi = \int\limits^w \sqrt q$.  
\end{proof}

\subsection{Exponential map and infinite-grafting} 
The proofs in the following subsection shall rely on the following notions that include a  well-known geometric interpretation of the exponential function (\textit{c.f.} \cite{Bisw2}).

Consider the entire function $f:\C \to \cp$ given by the exponential map $f(z) = e^{2\pi i z}$.   As is well known, this is the uniformizing function for the logarithm function $g(z) = \ln(z)$. Such an entire function has $0$ as an asymptotic value, as can be seen by restricting $f$ to the imaginary axis.  Let $R$ be an embedded arc between $0$ and $\infty$. The domain $\C$ can be thought of as obtained from $\cp$ by taking countably infinite copies of $\cp \setminus R$ indexed by $\Z$, and identifying one side of the slit $R$ on the $i$-th copy with the other side of the slit $R$ in the $(i+1)$-th copy. This procedure will be referred to as attaching a \textit{logarithmic end} between $0$ and $\infty$.   The point $0$  (or $\infty$) in $\cp$ is said to be a \textit{logarithmic singularity}, or equivalently, the map $f$ is a branched cover over $\cp$ with infinite ramification (or branching) over the branch points $0$ and $\infty$. Note that these branch-points are precisely the asymptotic values of the map $f$ in the domain. 

We remark that a  \textit{$2\pi$-grafting} along an embedded arc $\alpha \subset \cp$ is obtained by attaching a copy of $\cp \setminus \alpha$ along a slit at $\alpha$, or alternatively, grafting in a lune of angle $2\pi$ (see \S2.2).  Given a projective structure on a surface $S$, such an operation along an embedded arc in the developing image does not change the holonomy representation  (see \cite{GolProj}).

\subsection{Grafting and the Schwarzian derivative}

We verify that $\widehat{Gr}$ is well-defined, that is:

\begin{prop}\label{gr-schw}  The grafting operation on a crowned hyperbolic surface  in $\TT_g(\mathfrak{n})$ along a measured lamination $\lambda$ on it (as defined in \S3.4) results in a projective structure $P$  on a punctured Riemann surface $X$ with a Schwarzian derivative that lies in $ \mathcal{P}_g(\mathfrak{n})$. 
\end{prop}

\textit{Remark.} The fact that the grafting operation results in \textit{some} projective structure is a consequence of the definitions (see \S2.2); the above proposition identifies the space in which the resulting projective structure lies.\\

Let $\hat{S}$ be a crowned hyperbolic surface in $\TT_g(\mathfrak{n})$ . 
Recall that we have assumed, at the beginning of the section, that $\hat{S}$ has a single hyperbolic crown end with $(n-2)$ boundary cusps; we shall denote this crown by $\mathcal{C}$. (See \S3.3 for a description of a hyperbolic crown -- in particular, note that it is conformally an annulus of finite modulus.) 

The proof of Proposition \ref{gr-schw} is a local argument involving the grafting operation for this crown, and is an immediate consequence of the following two lemmas.

\begin{lem}\label{lem1-gr}
{The projective surface obtained by grafting $\hat{S}$ along $\lambda$ is conformally a punctured Riemann surface $X$.}
\end{lem}

\textit{Remark.} Note that $X$ is in fact of genus $g$ and a single puncture, that is, has the same topology as the crowned hyperbolic surface (see the second statement of Theorem \ref{thu-con}).

\begin{proof}[Proof of Lemma \ref{lem1-gr}] 

It suffices to check that the crown end $\mathcal{C}$ (that is topologically an annulus)  is conformally a punctured disk after the grafting; in the rest of the proof we shall focus entirely on this crown end.

Recall that the grafting lamination $\lambda$ intersects $\mathcal{C}$ along finitely many isolated leaves of finite weight, that are either between the boundary cusps of the crown, or from a boundary cusp to the closed geodesic boundary of $\mathcal{C}$. We denote the collection of these geodesic leaves of finite weight intersecting $\mathcal{C}$ by ${\lambda}_C$. More importantly for us, $\lambda$ includes the geodesic sides $\gamma_1, \gamma_2,\ldots \gamma_{n-2}$ of the crown boundary, each with infinite weight. \\

It shall be useful to pass to the universal cover. That is, consider the universal cover   $\widetilde{\mathcal{C}}$ of the hyperbolic crown as a $\mathbb{Z}$-invariant domain in  $\D \subset \cp$, and  perform a grafting along the lifted lamination $\widetilde{\lambda_C}$,  together with the $\mathbb{Z}$-invariant collection of the  lifts $\mathcal{G} = \{\tilde{\gamma}_1^i, \tilde{\gamma}_2^i, \ldots , \tilde{\gamma}_{n-2}^i\}_{i\in \mathbb{Z}}$ each with weight $\infty$.

The advantage is that the grafting here admits a  synthetic-geometric description similar to that in the previous section:

Namely, \\
(i) along each of the leaves of $\widetilde{\lambda_C}$ we insert a ``straight lune" of angle equal to the corresponding weight (see \S2.2), and \\
(ii) for each arc $\tilde{\gamma}\in \mathcal{G}$, we take countably infinite copies of $\cp$, each with a slit along $\tilde{\gamma}$, determining sides $\tilde{\gamma}_+^j$ and $\tilde{\gamma}_-^j$ on the $j$-th copy, where $j\in \mathbb{Z}_{\geq 0}$, and identify $\tilde{\gamma}$ (on the original domain) with $\tilde{\gamma}_+^1$, and then  successively identify $\tilde{\gamma}_-^j$ with $\tilde{\gamma}_+^{j+1}$ for each $j$. 

 (The fact that it is an infinite chain indexed by \textit{non-negative} integers, instead of a bi-infinite chain, is used in the proof of the final claim.) \\

Topologically, this appends an open half-disk along each boundary arc $\tilde{\gamma} \in \mathcal{G}$ in the original domain, and results in a conformal (immersed) domain  $\widetilde{A_\infty}$ in $\cp$.  See Figure 5.
As in \S2.2,   $\widetilde{A_\infty}$ is invariant under a cyclic subgroup of $\pslc$ generated by a new M\"{o}bius transformation, and yields  a conformal annulus $A_\infty$ in the quotient.

\begin{figure}
	\centering
	\includegraphics[scale=0.4]{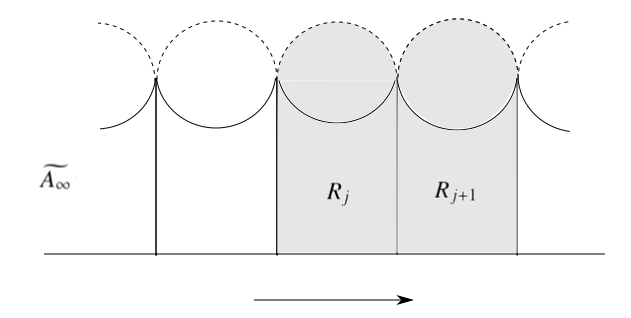}\\
	\caption{The surface $\widetilde{A_\infty}$ is obtained by infinite grafting along the circular arcs in the lift of the crown in $\cp$. This infinite grafting adds in a topological half-disk, denoted by the dotted lines.}
\end{figure}

 We have to show that the conformal modulus of $A_\infty$ is infinite, i.e.\ $A_\infty$ is biholomorphic to  $\D^\ast$.
 Equivalently, we need to show that $\widetilde{A_\infty}$ is biholomorphic to the upper half-plane.\\

Divide the strip $\widetilde{A_\infty}$  into topological rectangles by circular arcs from the crown-tips (starting points of  $\tilde{\gamma}_j^i$ for $j=1,2,\ldots (n-2)$ and $i\in \mathbb{Z}$. We denote these rectangles by $R_1^i,R_2^i,\ldots R_{n-2}^i$.

It suffices to show:\\

\textit{Claim. The conformal modulus of the union of any pair of adjacent rectangles in the above subdivision, call it $R_j \cup R_{j+1}$, is infinite.}\\
\textit{Proof of claim.} 
Here, we appeal to the grafting description for the exponential map described in \S4.2. Recall that here, a \textit{logarithmic end} comprising a bi-infinite chain of copies of $\cp \setminus \gamma $ is attached along a choice of an embedded arc $\gamma$ between the two branch-points of infinite order, $0$ and $\infty$, on $\cp$. Moreover, we know that the resulting Riemann surface is bi-holomorphic to $\C$.  Let  $D$ be a round disk in $\cp$ properly containing the arc $\gamma$ and its endpoints, and let $D^c = \cp \setminus D$ be the complementary disk. It follows  that the Riemann surface is obtained by  attaching the logarithmic end along $\gamma$ to $D$ is biholomorphic to $\C \setminus D^c$, which is conformally a punctured disk.

Observe that attaching a logarithmic end along $\gamma$ is equivalent to introducing a slit along $\gamma$, and then performing an infinite-grafting along the resulting two sides $\gamma_+$ and $\gamma_-$.  Here, we use the fact that an infinite-grafting along a side adds on a chain of copies of $\cp$ index by \text{non-negative} integers; thus, infinite-grafting along the two sides of the slit introduces a bi-infinite chain of $\cp$s, i.e.\ a logarithmic end. See Figure 6.

\begin{figure}[h]
	\centering
	\includegraphics[scale=0.47]{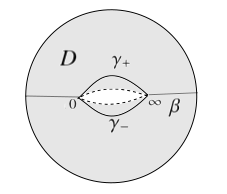}\\
	\caption{The surface obtained by infinite grafting along the two sides of a slit along $\gamma$ on a disk $D$. The rectangle $R^\prime = D \setminus \beta$ is quasiconformally related to $R_j \cup R_{j+1}$ (see Figure 5).}
\end{figure}

Pick two circular arcs from $0$ and $\infty$ respectively, to the boundary of $D$, intersecting $\partial D$ orthogonally. If we slit along one of the arcs, call it $\beta$,  we get a topological rectangle $R^\prime$, that is sub-divided into two rectangles $R_a$ and $R_b$ by the other arc. From the above discussion, the rectangle $R^\prime$ has infinite modulus, as the surface obtained by identifying the sides of the rectangle (the two sides of the slit $\beta$) is conformally a punctured disk.

Finally, note that one can easily build a quasiconformal map from $R_a \cup R_b$ to $R_j \cup R_{j+1}$; in fact, we can do so by a map that is conformal on the ends obtained by the infinite grafting on the sides.  Thus, $R_j \cup R_{j+1}$ is also a rectangle of infinite modulus, as claimed. 
$\qed$\\

This completes the proof of the Lemma.
\end{proof}

\medskip

Now let $U \cong \mathbb{D}^\ast$ be a neighborhood of the puncture on $X$ that corresponds to the crown end $\mathcal{C}$ after grafting.  In what follows we shall think of $U$ as a region $\{ \lvert z \rvert > 1\} \subset \C$.  The developing map $f:\widetilde{X} \to \cp$ for the projective structure $P$, when restricted to a lift $\tilde{U}$ of $U$, is a map  equivariant with respect to the action of $\pi_1(U) = \mathbb{Z}$ on the domain, and the cyclic monodromy around the puncture, in the target.

To complete the proof of Proposition \ref{gr-schw}, it suffices to show:

\begin{lem}\label{lem2-gr} 
{The Schwarzian derivative of $f\vert_{\tilde{U}}$ descends to a meromorphic quadratic differential on $U$ with a pole of order $n$ at the puncture.}
\end{lem}
\begin{proof}
Our proof is an adaptation of the ``rational approximation" argument of  Nevanlinna -- see \S3.4 of \cite{Nevanlinna}, and also the proof of Theorem 40.1 in \cite{Sib-book}.  

Consider the sequence of conformal annuli $A_N$ for $N\geq 1$ obtained by grafting $\mathcal{C}$ along $\lambda_C$, together with a  $2\pi N$-grafting on each of the geodesic sides $\gamma_1,\gamma_2, \ldots \gamma_{n-2}$ of the crown boundary.

It follows from the proof of Lemma \ref{lem1-gr} that $A_N$ form an exhaustion of $A_\infty$, that is,  $A_N \subset A_{N+1}$ for each $N$, such that $\text{mod}(A_N) \to \infty$ as $N\to \infty$ and  $A_\infty = \bigcup\limits_{N\geq 1} A_N$.

In particular, for any compact subset $\Omega \subset \tilde{U}$ there is a sufficiently large integer $N_0$ such that $\Omega$   is strictly contained in $\widetilde{A_N}$ for all $N\geq N_0$. Then the restriction $f\vert_{\Omega}$ is then the uniform limit of a subsequence of the corresponding developing maps $ f_N\vert_{\Omega}:\Omega \to \widetilde{A_N}$ where $N\geq N_0$. Recall that each $\widetilde{A_N}$ is conformally immersed in $\cp$, and by our construction $f_N$ is a conformal immersion to $\cp$ with order-$2N$ branching at the $\mathbb{Z}$-invariant collection of points where the lifts of two adjacent sides of the crown end meet. A simple calculation then shows that the Schwarzian derivative of $f_N$  is then of the form $\phi_N(z) dz^2$ where $\phi_N$ is a meromorphic function with poles of order at most two at the $(n-2)$ critical points that map to the branch-points of finite order. Thus, the restriction of $f$ to the interior of $A_N$, and in particular $f_N\vert_{\Omega}$ for $N\geq N_0$,  is a locally univalent holomorphic function since the critical points lie on the boundary of $A_N$. Moreover, since the number of poles of order two does not depend on $N$, this holomorphic function is of fixed polynomial growth that does not depend on $N$.

 By the uniform convergence $f_N \to f$ on $\Omega$,  these Schwarzian derivatives converge uniformly to the Schwarzian derivative of $f\vert_{\Omega}$, which is then of the form $\phi(z)dz^2$ where $\phi$ is a holomorphic function on $\Omega$ of a fixed polynomial growth that does not depend on $\Omega$.
 
By the usual invariance of the Schwarzian derivative under post-composition by M\"{o}bius maps,  this Schwarzian derivative  of $f\vert_{\tilde{U}}$ descends to a meromorphic quadratic differential on $U$. The polynomial growth condition then implies that it has at most a finite order pole at the puncture. 

The fact that the order of the pole is exactly $n$ follows from the discussion in \S4.1:

From our description of $\widetilde{A_\infty}$ in the proof of Lemma \ref{lem1-gr}, each fundamental region determines exactly $(n-2)$ infinitely-branched points, and thus the developing map $f\vert_{\tilde{U}}$ has exactly $(n-2)$ asymptotic values. From  Corollary \ref{cor-dev}, in the expression $\phi(z)dz^2$ for the Schwarzian derivative as expressed in $U = \{ \lvert z\rvert >1\}$, the rational function $\phi(z)$ would have polynomial growth of order exactly $(n-4)$, and thus the Schwarzian derivative has a pole of order $n$ at the puncture. 
\end{proof}

\subsection{Inverse of the grafting map}

Let $P$ be a meromorphic projective structure in $\mathcal{P}_g(\mathfrak{n})$, and let $\widetilde{P}$ denote the universal cover. 
Recall that the Thurston construction (see Theorem \ref{thu-con}) applied to $\widetilde{P}$ would yield the Poincar\'{e} disk $\mathbb{D}$ and a measured lamination $L$ on it. Recall that $L$  is the grafting lamination for $\widetilde{P}$.

In this subsection we shall prove the following Proposition, which says that the {\it image of the inverse of the grafting map lands in $\TT_g(\mathfrak{n})  \times {ML}_g(\mathfrak{n})$}.

\begin{prop}\label{inv}  For $\widetilde{P}$ as above,  the pair $(\mathbb{D}, {L})$ obtained from Theorem \ref{thu-con} is the universal cover of a pair $(X,\lambda) \in  \TT_g(\mathfrak{n})  \times {ML}_g(\mathfrak{n})$.
\end{prop}

\begin{figure}
  \centering
  \includegraphics[scale=0.5]{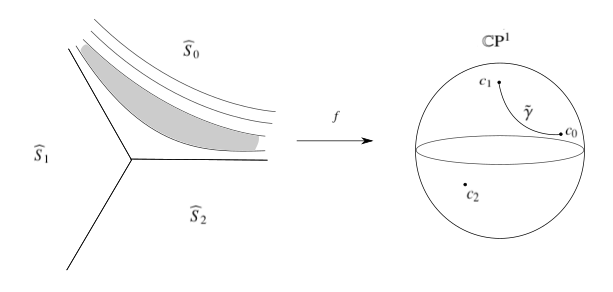}\\
  \caption{ A circular arc between a pair of consecutive asymptotic values of $f$ (see figure on right) has infinitely many pre-images in a sector (see figure on left). Each region between these pre-images (one shown shaded) map to a copy of $\cp \setminus \tilde{\gamma}$; the image of the sector thus wraps infinitely many times around $\cp$.}
\end{figure}

\begin{proof}

 For ease of notation, we shall continue with our assumption of a single puncture, in which case $\mathfrak{n}$ is just a single integer $n\geq 3$.

 Recall from Theorem \ref{thu-con} that by the equivariance of the developing map for  $\widetilde{P}$, it follows that $(\mathbb{D}, {L})$ would be invariant under some Fuchsian group $\Gamma$ such that $\mathbb{D}/\Gamma$ is homeomorphic to the underlying surface of $P$ -- a once-punctured surface of genus $g$.  

Restrict the projective structure $\widetilde{P}$ to the lift of a neighborhood $U$ of the puncture.
We need to verify that the  grafting lamination for the restriction $\widetilde{P \vert U}$  includes a cyclically ordered chain  of geodesics on $\D$ with infinite weight on each, such that the chain is invariant under a hyperbolic monodromy around the puncture.

By Corollary \ref{cor-dev}, the developing map for $\widetilde{P \vert U}$ descends to a meromorphic function $f$ on $U =\{ \lvert z\rvert >1\}$ having $(n-2)$ asymptotic values in equi-angled sectors $S_0, S_1,\ldots S_{n+1}$. We denote these asymptotic values by  $c_0, c_1,\ldots, c_{n+1} \in \cp$  respectively. Moreover, by Corollary \ref{cor-dev} the restriction of $f$ to each anti-Stokes sector of angle $2\pi/(n-2)$ has the same asymptotic expansion as an exponential map in suitable coordinates for the sector (see Equation \eqref{fexp}). In particular, the developing image of a sector is identical to that of the exponential map. See Figure 7.

By our geometric interpretation of the exponential map in \S4.1,  this developing image  can be described as follows:
for each $0\leq j\leq (n-3)$ choose a circular arc $\gamma_j$ in $\cp$ between  $c_j$ and $c_{j+1}$ such that  $\gamma_j$ is contained in the image of $f$. We obtain a Riemann surface $A_\infty$ by attaching a chain of copies of $\cp$ slit along $\gamma_j$, indexed by non-negative integers,  along each $\gamma_j$. We shall call this a \textit{half-logarithmic end}, which can be thought of as conformally immersed in $\cp$. The map $f$  then maps into $A_\infty$, and in particular, its restriction to a sector  surjects on to the corresponding half-logarithmic end. 

As a consequence of this geometric description for each pair of successive points $\{c_j,c_{j+1}\}$, there is a family of round disks embedded in $\cp$ parametrized by non-negative reals, such that each disk in the family touches $c_j$ and $c_{j+1}$, and their union exhausts the corresponding half-logarithmic end. In the immersed surface in $\cp$, this family of disks starts from a disk $D_0$ that has the circular arc $\gamma_j$ as part of its boundary, and then rotates around $\cp$, such that $D_t$ (where $t\in \mathbb{R}_{\geq 0}$) has a corresponding boundary arc that makes an angle $t$ with $\gamma_j$.  

The construction in Theorem \ref{thu-con} then shows that the corresponding pleated surface will have as bending line the geodesic line in $\mathbb{H}^3$ with endpoints $\{c_j,c_{j+1}\} \in \cp = \partial \mathbb{H}^3$. Moreover, in the construction of the associated pleated surface in Theorem \ref{thu-con}, the entire family of disks along the half-logarithmic end will collapse onto this line. In other words, the domain projective surface has an ``infinite" lune, and hence the corresponding leaf in the grafting lamination will have infinite weight. 

Passing to the universal cover, one obtains a chain of such geodesic lines in $\mathbb{D}$ that will be invariant with respect to the (Fuchsian)  holonomy around the puncture. To show that this monodromy is actually a hyperbolic element, we only need to rule out the case that it is parabolic, since we already know that $\mathbb{D}/\Gamma$ is  homeomorphic to the underlying surface $S$ of $P$:

Suppose the holonomy around the puncture is a parabolic transformation $h$. If the chain of geodesics $\{\tilde{\gamma}_i\}_{i\in \mathbb{Z}}$ in $\mathbb{D}$  is invariant under the infinite cyclic group $\langle h \rangle$, then their endpoints limit to the same point $p \in \partial \D$ as $i\to \pm \infty$, where the fixed point $\text{Fix}(h)=\{p\}$.  If we pick another element $g \in \pi_1(S)$ then the conjugate subgroup  $g\langle h\rangle g^{-1}$ would leave invariant another such chain of geodesics corresponding to another lift of a loop around the puncture. We note that the grafting lamination comprises disjoint leaves and $\text{Fix}(ghg^{-1}) = \{g.p\}$. If $g.p \neq p$, then  the two  chains of geodesics based at $p, \, g.p$ must intersect, contradicting the fact that no two leaves of the grafting lamination intersect.  Hence $g.p = p$, and since this is true for every element $g\in \pi_1(S)$, we conclude that $\Gamma$ is elementary, which is impossible as  $\mathbb{D}/\Gamma$ is homeomorphic to a surface with non-abelian fundamental group.

Thus, the above   chain of geodesics in $\D$ is invariant under this hyperbolic monodromy, and in the quotient $X=\D/\Gamma$, it descends to a crown end for the hyperbolic surface.  From our construction, in each fundamental region, there are exactly $(n-2)$ geodesic lines, and thus the crown end in the quotient has exactly $(n-2)$ boundary cusps.  Thus, the quotient hyperbolic surface $X$ lies in  $\TT_g(\mathfrak{n})$.

Moreover, the grafting lamination  $L$ on $\D$ is invariant under $\Gamma$, and descends to a measured lamination $\lambda$ on such a crowned hyperbolic surface, and thus, by definition, lies in ${ML}_g(\mathfrak{n})$ (\textit{c.f.} \S3.4).
\end{proof}

\medskip

We can finally show:

\begin{prop} The grafting map $\widehat{Gr}$ is a homeomorphism.
\end{prop}
\begin{proof}
Recall from Theorem \ref{thu-con} that the Thurston construction for a projective surface obtained by grafting recovers the original  hyperbolic surface and measured geodesic lamination.

By Proposition \ref{inv}, the Thurston construction then defines an inverse map to $\widehat{Gr}$. Moreover, by the same proposition, 
 $\widehat{Gr}$ is surjective. 
 
Since the domain of the map is  homeomorphic to $\R^{2\chi}$ (see Proposition \ref{mln}) and $\widehat{Gr}$ is continuous (see the last statement of Theorem \ref{thu-con}, we conclude that $\widehat{Gr}$ is a homeomorphism (by invariance of domain). \end{proof}

This completes the proof of Theorem 1.1.

\section{Projective structures on $\C$}

The proof of Theorem \ref{thm1} also applies in the case when $g=0$ and $k=1$, and we obtain a grafting description for a certain space of projective structures on the complex plane $\mathbb{C}$ -- see Theorem \ref{thm2} from \S1. 
After defining the spaces appearing in Theorem \ref{thm2} in \S5.1, we provide a proof, and give an application of Theorem \ref{thm2} in \S5.2.

\subsection{Definitions and the proof of Theorem \ref{thm2}} 
We start with a more detailed description of the spaces in Theorem \ref{thm2}:\\

A projective structure in $\mathcal{P}(d)$ is determined by a conformal immersion $f:\mathbb{C} \to \cp$  (the developing map) such that the Schwarzian derivative of $f$ (see Equation \eqref{schwd})  is a polynomial quadratic differential on $\mathbb{C}$ of degree $d$, that is, it can be expressed as 
\begin{equation*}
q = (z^d + a_{d-2}z^{d-2} + \cdots + a_1z + a_0)dz^2
\end{equation*}
 where the coefficients $(a_0,a_1, \ldots, a_{d-2}) \in \C^{d-1}$.   Note that, up to a conformal automorphism of $\mathbb{C}$, any polynomial quadratic differential can be assumed to be monic and centered as above. 
 
 In this section, there will be no additional real twist parameter at $\infty$; indeed, there are no non-trivial Dehn-twists around $\infty$ since $\C$ is simply-connected, and the normalization as above fixes the horizontal directions of $q$ to be at angles $2\pi j/(d+2)$ where $j=0,1,\ldots (d+1)$.\\ 
 
 The existence and uniqueness of such projective structures is a consequence of the work of Sibuya -- see \S5.2  for a discussion. 
Moreover, it follows from his work (see Corollary \ref{cor-dev}) that the entire function $f$ has exactly $(d+2)$ asymptotic values that we call the \textit{crown tips}.  As usual, we shall consider two projective structures on $\C$  to be \textit{equivalent } if the developing maps are isotopic  such that the isotopy keeps the crown tips fixed. 
Recall from Corollary \ref{cor-dev} that the asymptotic values are achieved along rays in the horizontal directions of $q$ which are at equal angles of $2\pi/(d+2)$ starting from the horizontal direction; this gives a cyclic ordering to the set of crown tips.\\

Sibuya showed (see Chapter 8 of \cite{Sib-book}), using the methods from the theory of linear differential systems , that in fact the cyclically  ordered collection of (possibly non-distinct) crown-tips $\mathsf{C} = \{ c_0, c_1, \cdots c_{d+1} \}$ on $\mathbb{C} \mathrm{P}^1$  satisfy (a) $c_k \neq c_{k+1}$ and (b) there are at least three distinct points in $\mathsf{C}$. 

Let $\mathfrak{C}(d)$ be the space of ordered $(d+2)$-tuples in $\cp$ that satisfy (a) and (b) above, up to the action of $\pslc$. (In particular, we can arrange so that the first three points are $0,\infty$ and $1$.) 

We can define the ``crown-tip map"
\begin{equation}\label{psimap}
\Psi: \mathcal{P}(d)  \to \mathfrak{C}(d)
\end{equation}
that assigns to a projective structure on $\C$, the ordered tuple of crown-tips that it determines.\\

Next, $\text{Poly}(d)$ is the space of hyperbolic ideal polygons with $(d+2)$ vertices $a_0, a_1, \cdots, a_{d+1}$ up to isometry, together with a cyclic ordering of  the vertices. Assume, without loss of generality, that $a_0, a_1, \cdots, a_{d+1}$ gives this cyclic ordering. Suppose further that after acting by a suitable isometry, the vertices $a_0, a_1, a_2$ are placed at $ -1, 1, i$. The cross-ratios of successive quadruples $\{a_j, \cdots, a_{j+3}\}$,
$j=0,1, \cdots, d-2$ for the remaining ideal vertices determine $(d-1)$ real parameters that uniquely determine the ideal polygon. Thus, the space $\text{Poly}(d)$ is homeomorphic to $\R^{d-1}$.\\

Finally, the space $\text{Diag}(d)$ is the space of weighted diagonals in an ideal $(d+2)$-gon, where each of the (cyclically ordered) geodesic sides of the polygon have infinite weight. As in the proof of Proposition \ref{lamcrown}, it is useful to consider the corresponding space of dual metric trees, where the length of an edge equals the weight of the diagonal it represents.  It is well-known that the space of such dual metric trees is homeomorphic to $\R^{d-1}$ -- see, for example, Theorem 3.3 of \cite{MulPenk} and the discussion in section 3.2 of \cite{GupWolf3}. (Note that the geodesic sides of infinite weight do not contribute any parameters.) \\

We shall assume that each of these spaces acquire a natural topology via the parametrization we have described for them.

\smallskip

\begin{proof}[Proof of Theorem \ref{thm2}]

The fact that the grafting map described in Equation \eqref{graft3}, which is:
\begin{equation*}
 		\widehat{Gr}_{\C}: \text{Poly}(d)  \times \text{Diag}(d) \to  \mathcal{P}(d)
 		\end{equation*}
is  well-defined follows from Lemmas \ref{lem1-gr} and \ref{lem2-gr} in \S4.3. \\

This is also implied by the work of Sibuya in \cite{Sib-book}, as we now describe:

Let $P\in \text{Poly}(d)$ be an ideal polygon, thought of as conformally embedded in $\D \subset \cp$,  with ideal vertices  $a_0,a_1,\ldots, a_{d+1}$ along the equatorial (real) circle. It is easy to verify that grafting $P$ along a set of diagonals with finite weight takes these vertices to an ordered tuple of points $c_0, c_1,\ldots c_{d+1}$ that lies in the space $\mathfrak{C}(d)$ defined above.

Given such an ordered set $\mathsf{C}$ of points in $\cp$ satisfying (a) and (b) above, Sibuya considered the Riemann surface $\mathcal{R}$ by attaching an infinite chain of copies of $\cp$ (\textit{c.f.} \S4.2) to arcs chosen between successive points.  In our grafting terminology, this is equivalent to performing, in addition to the grafting along the diagonals in $P$ of finite weight,  an infinite grafting along the geodesic sides of $P$. 

A theorem of  Nevanlinna  (\cite{Nevan}) then asserts  that the resulting surface $\mathcal{R}$  is parabolic, i.e.\  $\mathcal{R}$  is conformally equivalent to $\mathbb{C}$. (This is the analogue of  Lemma \ref{lem1-gr} from \S4.3.)  Moreover, Theorem 40.1 of \cite{Sib-book} shows that the map  $f:\mathbb{C} \to \cp$, i.e.\ the composition of the biholomorphism from $\mathbb{C}$ to $\mathcal{R}$, followed by the branched cover to $\cp$, has a Schwarzian derivative that is a polynomial quadratic differential of degree $d$.  (This is the analogue of Lemma \ref{lem2-gr} from \S4.3.)  By construction, the asymptotic values of $f$ are the infinite-order branch-points at $\mathsf{C}$.  Thus, $f$ defines a projective structure $P\in \mathcal{P}(d)$, with the crown-tips $\mathsf{C}$. See Chapter 8 \S40, 41 of \cite{Sib-book} for details. \\

Then, the proof in \S4.4  carries through, to show that $\widehat{Gr}_{\C}$ admits an inverse map. Recall that this uses the Thurston construction -- see Proposition \ref{inv}. In fact, the present discussion would be easier than the work required in the proof of Proposition \ref{inv}, since the punctured surface $\C$ is simply-connected, and  we need not pass to the universal cover. Theorem \ref{thu-con} applies directly, and the argument in Proposition \ref{inv} (that uses Corollary \ref{cor-dev}) shows that the grafting lamination includes a closed chain of $(d+2)$ geodesic lines in $\D$,  each of infinite weight, that thus bounds an ideal polygon $P\in \text{Poly}(d)$. The remaining geodesic leaves of the grafting lamination must be pairwise disjoint, and hence must constitute a collection of weighted diagonals in $P$.

Thus, this inverse map has image in $\text{Poly}(d)  \times \text{Diag}(d)$ when we start with any projective structure in  $\mathcal{P}(d)$.

In particular, this proves that $\widehat{Gr}_{\C}$ is a bijection. Since the spaces in the domain and range of $\widehat{Gr}_{\C}$ in Equation \eqref{graft3} are homeomorphic to $\R^{2d-2}$, we conclude, from the invariance of domain, that $\widehat{Gr}_{\C}$ is a homeomorphism. 
\end{proof}

\subsection{Fibers of the crown-tip map} 
In this section, we use the grafting description in Theorem \ref{thm2} to characterize the fibers of the map $\Psi$ in Equation \eqref{psimap}, i.e.\ the set of all projective structures in $\C$ that have the same ordered set of crown-tips (as defined in \S5.1).\\

The work of Bakken in \cite{Bakk} showed that $\Psi$ is in fact a local biholomorphism.
However, it was known, due to examples of Sibuya (see \S42 of \cite{Sib-book}) and Bakken (see \S7 of \cite{Bakk}), that $\Psi$ is not globally injective.

We shall now prove:

\begin{thm}\label{nonuniq}
Fix an ordered tuple $\mathsf{C} \in \mathfrak{C}(d)$. 
For any disjoint collection of diagonals 
\begin{equation}\label{gr0}
\mathcal{D}= \{l_1,l_2,\ldots, l_{d-1}\}
\end{equation}
 in an abstract $(d+2)$-gon, there exists a unique ideal polygon  $P\in \text{Poly}(d)$ and a unique collection of non-negative weights $\{w_1,w_2,\ldots, w_{d-1}\}$,
 $w_i \in [0, 2\pi)$, on the diagonals such that 
\begin{equation}\label{gr1}
\widehat{Gr}_{\C}\left(P,  L\right) \in \Psi^{-1}(\mathsf{C})
 \end{equation}
 whenever $L\in \text{Diag}(d)$ is a weighted diagonal assigning  weight $(w_i+ 2\pi n_i)$ to the diagonal $l_i$, for a tuple $(n_1,n_2,\ldots n_{d-1}) \in \mathbb{Z}_{\geq 0}^{d-1}$, together with the geodesic sides of $P$, each with infinite weight.  We write this as:
 \begin{equation}\label{gr2}
 L = (w_1+ 2\pi n_1)\cdot l_1 + (w_2+ 2\pi n_2)\cdot l_2  + \cdots + (w_{d-1}+ 2\pi n_{d-1})\cdot l_{d-1}
 \end{equation}
Moreover, any  element of the fiber $\Psi^{-1}(\mathsf{C})$ is given via a grafting construction (Equation \eqref{gr1}) by the following data:
\begin{enumerate}
\item a choice of diagonals as in Equation \eqref{gr0},
\item the unique associated ideal polygon $P$ and $(d-1)-$tuple of weights as above, with
$w_i \in [0, 2\pi)$, 
\item a $(d-1)-$tuple of integers $n_i \in \Z_{\geq 0}$ as in Equation \eqref{gr2}.
\end{enumerate}
\end{thm}

\medskip

Before giving the proof of Theorem \ref{nonuniq}, we describe the operation of grafting for an ideal quadrilateral that will play a role; note in particular that grafting ideal polygons along a collection of weighted diagonals can be described completely in two dimensions, that is, on the complex plane $\C$, without reference to three-dimensional hyperbolic geometry as in \S2.2.

\subsubsection*{Grafting an ideal quadrilateral} Consider an ideal quadrilateral defined by the (cyclically ordered) tuple of ideal vertices $\infty, -1, 0, \lambda$ where $\lambda \in \mathbb{R}^+$. A grafting (or ``bending") by angle $t$ along the diagonal between $0$ and $\infty$ can be seen on the upper half-plane as follows: the diagonal line in this model is the vertical geodesic $\alpha$ from $0$ to $\infty$; this divides the upper half-plane into the two regions $R_-$ and $R_+$  that are the quarter-planes defined by $\text{Re}(z) <0$ and $\text{Re}(z) >0$ respectively. The grafting is then effected by a map that is the identity on $R_-$ and the rotation $z\mapsto e^{-i\theta}$ on $R_+$; the image is a new domain that is obtained from the upper half-plane by grafting in a lune of angle $t$, at the vertical geodesic $\alpha$.  Clearly, the grafting fixes the points $-1,0,\infty$ and takes $\lambda$ to the new point $\lambda e^{-it} \in \cp$.  (See Figure \ref{lune} in \S2.2.) 

Note that the resulting tuple of points $(\infty, -1, 0, \lambda e^{-it})$ could also have been obtained by grafting along the diagonal line $\alpha^\prime$ between $-1$ and $\lambda$.  A cross-ratio calculation shows that there is a conformal map that realizes the permutation $(\infty, -1, 0, r) \mapsto (-1, 0, 1/r, \infty)$, and hence a graft by an angle $2\pi - t$ along $\alpha^\prime$, results in the same configuration of four points.

\begin{proof}[Proof of Theorem \ref{nonuniq}]

Let $\{c_0,c_1,\ldots, c_{d+1}\}$ be the ordered tuple $\mathsf{C} \in \mathfrak{C}(d)$.

An ideal $(d+2)$-gon  with ideal vertices $a_0,a_1,\ldots a_{d+1}$  would be triangulated by the collection of diagonals determined by $\mathcal{D}$.  Each diagonal line $\delta$ determines an ideal quadrilateral $Q$ comprising the two ideal  triangles adjacent to  $\delta$. This determines  a collection $\mathcal{Q}$  of overlapping quadrilaterals:  each pair of quadrilaterals  in $\mathcal{Q}$ is either disjoint, or overlaps along an ideal triangle. Note that there is a dual tree $T$ determined by this configuration of diagonals -- the vertices of $T$ correspond to the ideal quadrilaterals, and there is an edge between vertices whenever the corresponding quadrilaterals overlap. 

 It is easy to check by an inductive proof based on the tree $T$, that the ideal $(d+2)$-gon is uniquely determined by the cross ratios of the quadrilaterals in $\mathcal{Q}$, where vertices of each are taken in the induced cyclic order.
 
Now for each quadrilateral $Q \in \mathcal{Q}$  we can choose the ideal vertices $\{a_j,a_k,a_l,a_m\}$ of $Q$ such that it has a cross-ratio $\lvert \lambda_Q \rvert$, where $\lambda_Q$ the cross-ratio of the four points $c_j,c_k,c_l,c_m$. 
Let $P$ be the ideal polygon that this data  uniquely determines.

To assign weights to these diagonals, note that  $Q$ has a diagonal $d_Q \in \DD$; the toy example preceding the lemma describes how one can graft $Q$ along this diagonal $\delta$ by an angle $w(Q) \in [0,2\pi)$ such that the images of the vertices  $\{a_j,a_k,a_l,a_m\}$ are the four points  $c_j,c_k,c_l,c_m$ (in the ordered tuple $\mathsf{C}$) with cross-ratio equal to $\lambda_Q$. 
We equip that  diagonal $d_Q$ with weight $w(Q)$.

Thus by construction, grafting each diagonal $d_Q$ in $\DD$ by an angle $w(Q)$, we obtain $\widehat{Gr}_{\C}\left(P,  L\right)$ where (see Equation \eqref{gr2}) $$L = \sum\limits_{Q} w(Q)d_Q$$ takes the vertices of $P$ to the tuple of points $c_0,c_1,\ldots c_{d+1}$, as desired.  Thus, $\widehat{Gr}_{\C}\left(P,  L\right)$ is a projective structure in $\mathcal{P}(d)$ with crown-tips exactly the ordered tuple $\mathsf{C} \in \mathfrak{C}(d)$.  (The infinite grafting on the geodesic sides of $P$ does not affect the positions of these crown-tips.) 

\begin{figure}
  \centering
  \includegraphics[scale=0.53]{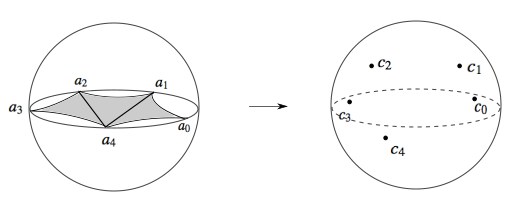}\\
  \caption{Any configuration of points  on $\cp$(right) can be obtained by grafting an ideal polygon along a collection of weighted diagonals (left).}
\end{figure}

Note that adding $2\pi$ to the weights of the diagonals, i.e.\ performing integer-$2\pi$ grafting (see \S4.1) does not change the  configuration of crown tips $\mathsf{C}$.

On the other hand, by Theorem \ref{thm2}, any projective structure in $\Psi^{-1}(\mathsf{C})$ is of the form  $\widehat{Gr}_{\C}\left(P^\prime,  L^\prime\right)$ for some ideal polygon $P\in \text{Poly}(d)$ and weighted diagonals $L^\prime \in \text{Diag}(d)$.  By the previous discussion, the collection of diagonals $\mathcal{D}$ underlying $L^\prime$ uniquely  determines $P^\prime$, and the weights of $L^\prime$ modulo $2\pi$. \end{proof}

\section{The monodromy map and Theorem \ref{thm3}}

In this final section we consider the monodromy map
\begin{equation}\label{mmap}
\Phi : \mathcal{P}_g(\mathfrak{n}) \to {\widehat{\bigchi}}_{g,k}(\mathfrak{n}) 
\end{equation}
where the target is the \textit{decorated character variety} that we shall define in \S6.1.

In \S6.2, we shall prove Theorem \ref{thm3}; this shall use the grafting description for meromorphic projective structures that Theorem \ref{thm1} provides.

\subsection{Decorated character variety}

For an oriented surface $S_{g,k}$ of genus $g$ and $k$ (labeled) punctures and negative Euler-characteristic, the usual $\pslc$-\textit{character variety} is
\begin{equation*}
\bigchi_{g,k} = \text{Hom}(\pi_1(S_{g,k}), \pslc)/\!\!/\pslc
\end{equation*}
where the geometric-invariant-theory (GIT) quotient on the right, yields a quasi-projective variety of (complex) dimension $6g-6 + 3k$.

In what follows, we shall denote the {\it representation variety} as $$\mathcal{R}_{g,k} :=\text{Hom}(\pi_1(S_{g,k}), \pslc).$$ Thus, $\mathcal{R}_{g,k}$ is the space of representations, prior to the quotient. Given $\rho \in \mathcal{R}_{g,k}$, the monodromies around the $k$ punctures shall be denoted by  $\rho_1,\rho_2,\ldots, \rho_k \in \pslc$  respectively. \\

Fix a $k$-tuple $\mathfrak{n} = (n_1,n_2,\ldots, n_k)$ where each $n_i \geq 3$.

Recall from \S3.1 that a meromorphic projective structure  $P\in \mathcal{P}_g(\mathfrak{n})$ is a projective structure on a 
surface $\hat{S}$ of genus $g$ and $k$ boundary components, with $m_i:=n_i-2$ marked points on the $i$-th boundary component, where $1\leq i\leq k$. 
In particular, the holonomy of the projective structure determines a representation $\rho \in \bigchi_{g,k}$.  

In addition to this, we know from the grafting description provided by Theorem \ref{thm1}, or from Corollary \ref{cor-dev}, that the developing map for $P$, when restricted to neighborhood of the $i$-th pole of order $n_i\geq 3$, has $m_i:=(n_i-2)$ asymptotic values, where $1\leq i\leq k$.
This yields a point in $\cp$ for each connected component of a $\partial \hat{S} \setminus \{\text{marked points on the boundary}\}$. 

Passing to the universal cover, we have a family of points on $\cp$ that are translates of an (ordered) fundamental set  $\mathsf{C}_i = \{c_0^i,c_1^i,\ldots c_{m_i-1}^i\}$, by the monodromy $\rho_i$ around the $i$-th boundary component of $\hat{S}$. 

Just as for the set of crown-tips $\mathfrak{C}(d)$ (defined before Equation \ref{psimap}), no two adjacent points in $\mathsf{C}_i$ are the same, that is, $c_j^i \neq c_{j+1}^i$ for $0\leq j\leq m_i-2$.   (Their translates under $\rho_i$ may coincide, for example, when $\rho_i$ is an elliptic element of finite order.)

Let $\widehat{\mathfrak{C}}(m_i, \rho_i)$ be the space of such ordered $m_i$-tuples of points in $\cp$, together with a choice of a (fixed) monodromy matrix $\rho_i\in \pslc$ (which determines translates of the ordered set by the cyclic group generated by $\rho_i$).
This defines the space of \textit{decorations} at a boundary component of $\hat{S}$, and its quotient by the conjugation-action of $\pslc$  is exactly the space of ``{configurations of flags}"  defined by Fock-Goncharov (see pg. 11 of \cite{FG}), in the present context,  since a ``flag" in $\C^2$ can be thought of as a point in $\cp$.   \\

We then define:

\begin{defn}\label{decv} 
The \textit{decorated character variety} is the space
\begin{equation*}
{\widehat{\bigchi}}_{g,k}(\mathfrak{n})  = \left\{ (\rho, \mathsf{C}_1, \mathsf{C}_2, \ldots, \mathsf{C}_k) \text{ } \bigg\vert\text{ } \rho \in \mathcal{R}_{g,k} \text{ and } \mathsf{C}_i  \in \widehat{\mathfrak{C}}(m_i, \rho_i) \text{ for } 1\leq i\leq k \right\}\bigg/\!\bigg/\pslc 
\end{equation*}
where  $m_i = n_i-2$ and $\rho_i$ is the monodromy around the $i$-th puncture, for each $1\leq i\leq k$. 
\end{defn}

\textit{Remarks.} 1. This coincides with the notion of the \textit{moduli stack of framed representations (or framed local systems)} of Fock-Goncharov -- see Definition 2.2  of \cite{FG} or Definition 2.7 of \cite{Palesi}, and also \S1.4 of \cite{AllBrid}. \\
2. For other notions of a ``decorated" character variety, see \cite{Boalch} or \cite{Chekhov} (see \S1.6.3 of \cite{AllBrid} for a discussion.) \\

From the preceding discussion, the meromorphic projective structure  $P\in \mathcal{P}_g(\mathfrak{n})$ uniquely determines a \textit{decorated monodromy} $\hat{\rho} \in {\widehat{\bigchi}}_{g,k}(\mathfrak{n})$. This defines the monodromy map $\Phi$ (see  \eqref{mmap}). \\

Moreover,  the work of Allegretti-Bridgeland shows that:\\
(a)  the image of the monodromy map $\Phi$ lies in an open dense subset ${\widehat{\bigchi}}_{g,k}(\mathfrak{n})^\ast \subset {\widehat{\bigchi}}_{g,k}(\mathfrak{n}) $ comprising (in their terminology) those representations having non-degenerate framing -- see \S6 of \cite{AllBrid},  \\
(b) the space ${\widehat{\bigchi}}_{g,k}(\mathfrak{n})^\ast$ is a complex manifold -- see \S9 of \cite{AllBrid}. \\

The following is essentially a consequence of  the ``Decomposition Theorem" of Fock-Goncharov (see Theorem 1.1. of  \cite{FG}, and Theorem 2.8 of \cite{Palesi}).

\begin{prop}\label{dimchar} The space ${\widehat{\bigchi}}_{g,k}(\mathfrak{n})^\ast$ is a complex manifold of dimension $\chi =6g-6 + \sum\limits_{i=1}^k (n_i+1)$.
\end{prop}

\begin{proof}

By (b) above, it is enough to verify there is an open set of \textit{real} dimension $\R^{2\chi}$ contained  in   ${\widehat{\bigchi}}_{g,k}(\mathfrak{n})^\ast$. 

Pick a crowned hyperbolic surface $X\in \mathcal{T}_g(\mathfrak{n})$ and a measured lamination in $\mathcal{ML}_g(\mathfrak{n})$ that comprises a maximal set $\mathcal{L}$ of disjoint weighted bi-infinite geodesics between the various boundary cusps. Note that the maximality implies that $\mathcal{L}$  divides $X$ into ideal triangles, and an easy combinatorial count using the Euler-characteristic of the punctured surface implies that the cardinality $\lvert \mathcal{L} \rvert = \chi$. In fact, from Proposition \ref{mln} we obtain an open set  $V \subset \mathcal{ML}_g(\mathfrak{n})$ containing $\mathcal{L}$ by varying the weights of the geodesic lines underlying $\mathcal{L}$. Moreover, it is easy to check,  from shear-coordinates in Teichm\"{u}ller theory (see, for example, \cite{BBFS}), that by varying the real ``shear" parameter on each of the geodesic lines in $\mathcal{L}$, we obtain an open set  $U \subset \mathcal{T}_g(\mathfrak{n})$ containing $X$. 

 By Theorem \ref{thm1}, we know $P^\prime = \widehat{Gr}(X^\prime , \mathcal{L}^\prime) \in \mathcal{P}_g(\mathfrak{n})$ for any pair $(X^\prime, \mathcal{L}^\prime) \in U\times V$, and by part (a) of the discussion preceding the theorem, the decorated monodromy $\hat{\rho}$ for such a  structure is a point in ${\widehat{\bigchi}}_{g,k}(\mathfrak{n})^\ast$.
 
 By Propositions \ref{tgn} and \ref{mln}, both $U$, $V$ are homeomorphic to $\R^\chi$, and by (b) above ${\widehat{\bigchi}}_{g,k}(\mathfrak{n})^\ast$ is a manifold. Thus, it suffices, by the invariance of domain, to show that $\Phi \circ \widehat{Gr}$ is injective on $U\times V$. \\

 Let $\rho \in \bigchi_{g,k}$ be the  holonomy of $P^\prime$ forgetting the decorations (\textit{c.f.} Definition \ref{decv}).
  By the construction in the proof of Theorem \ref{thu-con} we know that via the developing map for this projective structure, we obtain a $\rho$-equivariant map $\Psi:\widetilde{X} \to \mathbb{H}^3$. The map  $\Psi$ gives a pleated plane $\mathcal{P}$ that is pleated (or bent) at the lifts of the leaves of $\mathcal{L}^\prime$. Note that $\Psi$ is well-defined up to post-composition by an element of $\pslc$. By the maximality of $\mathcal{L}^\prime$, these lifts triangulate $\mathcal{P}$ into a $\rho$-invariant collection of ideal triangles. Recall that after ``straightening" $\mathcal{P}$, we obtain a totally geodesic copy of the hyperbolic plane, and the pleating locus determines a  lamination invariant under a Fuchsian group $\Gamma$. 
 
As in the proof of Theorem \ref{nonuniq}, the collection of geodesic lines $\widetilde{\mathcal{L}}^\prime$ determines a collection of ``bent" quadrilaterals $\mathcal{Q}$ by assigning, to each line $l\in \widetilde{\mathcal{L}}^\prime$, the quadrilateral $Q$ formed by the two ideal triangles in $\mathcal{P}$ adjacent to $l$. \\

Theorem 1.1 of \cite{FG}, applied to our setting ($G =\pslc$) implies that the  complex cross ratios of each $Q\in \mathcal{Q}$ is determined by the decorated monodromy $\hat{\rho}$. (See also the Example on page 11 of \cite{FG}.)  Then, from the discussion in \S5.2 on grafting an ideal quadrilateral, the weights on the leaves of $\widetilde{\mathcal{L}}^\prime$, as well the real cross-ratios of the ``straightened" quadrilaterals, are determined uniquely by these complex cross-ratios.  The real cross-ratios, in turn, determine the ideal quadrilaterals (overlapping along ideal triangles) that constitute the fundamental domain of the $\Gamma$-action on the ``straightened" pleated plane. This uniquely determines the hyperbolic surface $X^\prime$ that we obtain in the quotient. Moreover, the $\Gamma$-invariance of  the pleating locus determined by $\widetilde{\mathcal{L}}^\prime$ shows that the weighted geodesics constituting the lamination $\mathcal{L}^\prime$ are uniquely determined.

In other words, the decorated monodromy $\hat{\rho}$ recovers the pair $(X^\prime, \mathcal{L}^\prime) \in U \times V$. 
This completes the proof of the injectivity of the monodromy map $\Phi \circ \widehat{Gr}$ on $U\times V$, and hence of the Proposition. \end{proof}

 \textit{Remarks.} 1.  Alternatively, one can show that at a generic point, the space of configurations $\widehat{\mathfrak{C}}(m_i, \rho_i)$, for fixed $\rho_i$,  is of complex dimension $m_i$, for each $1\leq i\leq k$. Adding these contributions to $\text{dim}_{\C}(\bigchi_{g,k}) = 6g-6+3k$, we again get complex dimension $\chi$.\\
 2. Our proof of Proposition \ref{dimchar} in fact shows that the monodromy map $\Phi$ is a homeomorphism on the open set $U\times V$. The argument in the next section to show that $\Phi$ is a local homeomorphism \textit{everywhere} will follow a similar strategy, with an additional difficulty arising from the fact that the grafting lamination may not be just a maximal set of weighted geodesic lines.

\subsection{Proof of Theorem \ref{thm3}} 

In this section, we shall prove Theorem \ref{thm3}, namely, that $\Phi$ in  \eqref{mmap} is a local homeomorphism.\\

Throughout the section, we shall fix a base meromorphic projective structure $P \in \mathcal{P}_g(\mathfrak{n})$ that has monodromy $\widehat{\rho} \in {\widehat{\bigchi}}_{g,k}(\mathfrak{n})$.

By Theorem \ref{thm1}, we know that $P =\widehat{Gr}(X,\lambda)$ for some pair  $(X, \lambda) \in  \mathcal{T}_g(\mathfrak{n})\times   \mathcal{ML}_g(\mathfrak{n})$.

Our task is to show that there is a small neighborhood $U$ of $X$ in $\mathcal{T}_g(\mathfrak{n})$ and a neighborhood $V$ of $\lambda$ in $\mathcal{ML}_g(\mathfrak{n})$, such that if 
\begin{equation}\label{meq} 
\Phi \circ \widehat{Gr}(X^\prime , \lambda^\prime)  = \Phi \circ \widehat{Gr}(X , \lambda)  = \widehat{\rho}  
\end{equation}
for a pair $(X^\prime,  \lambda^\prime) \in U \times V$, then we have 
\begin{equation}\label{feq} 
X=X^\prime \text{ and } \lambda= \lambda^\prime.
\end{equation}

\medskip

Let $X= X_S \cup \mathcal{C}$ where $X_S$ is a hyperbolic surface with geodesic boundary components, and $\mathcal{C}$ is the collection of crown ends. Similarly, we have the decomposition of any crowned hyperbolic surface $X^\prime = X_S^\prime  \cup \mathcal{C}^\prime$.

Given a measured lamination $\lambda$ on $X$, let $\lambda = \lambda_S \cup \mathcal{L}$ where $\lambda_S$ is supported in a compact part of the surface  $X_S$ (away from the crown ends), and $\mathcal{L}$ consists of the finitely many leaves of $\lambda$ that intersect the crown ends $\mathcal{C}$. (Here, we shall ignore the geodesic sides of the crowns, each of which have infinite weight.)

Similarly, we have the disjoint union $\lambda^\prime = \lambda_S^\prime \cup \mathcal{L}^\prime$ on the crowned hyperbolic surface $X^\prime$. \\

In what follows we shall call $\mathcal{L}$ (resp. $\mathcal{L}^\prime$) a \textit{triangulation} of the crowned surface $X$ (resp. $X^\prime$), if there is no geodesic line between the boundary cusps of $\mathcal{C}$ (resp. $\mathcal{C}^\prime$) that is disjoint from the leaves already in the collection.  Note that if $\mathcal{L}$ (resp. $\mathcal{L}^\prime$) is \textit{not} a triangulation, then we can choose an extension to a triangulation by adding geodesics of zero weight between boundary cusps of the crowns. We shall denote the resulting triangulation by $\mathcal{L}_+$  (resp. $\mathcal{L}_+^\prime$).\\

The first step of the proof is to show:

\begin{prop}\label{prop1} Suppose Equation \eqref{meq} holds, where the pair $(X^\prime, \lambda^\prime)$ is sufficiently close to $(X,\lambda)$ in   $\mathcal{T}_g(\mathfrak{n})\times   \mathcal{ML}_g(\mathfrak{n})$, then the crown ends of $X^\prime$ and $X$ are isometric, and $\mathcal{L}^\prime = \mathcal{L}$.
\end{prop}

\begin{proof}

Choose a neighborhood $V_0$ of $\lambda$ in $ \mathcal{ML}_g(\mathfrak{n})$ such that for any $\lambda^\prime \in V_0$, 
there are triangulations $\mathcal{L}^\prime_+$ (resp. $\mathcal{L}_+$) of $\mathcal{C}^\prime$ (resp. $\mathcal{C}$), such that the homotopy classes of the arcs in the triangulations are identical, and the corresponding weights are close. (For the notation used here, see the paragraph preceding this Proposition.)  

This is possible since the space  $ \mathcal{ML}_g(\mathfrak{n})$ is a cell-complex where the finitely many cells correspond to the different topological types of the dual metric graphs (\textit{c.f.} the proofs of Propositions \ref{lamcrown} and  \ref{mln}). The homotopy classes of the  arcs in the triangulation determines this topological type of the dual metric graph; once this is fixed, two laminations being close in  $ \mathcal{ML}_g(\mathfrak{n})$ implies that the corresponding weights on the arcs (that determine the lengths of the finite edges) are close. 

Throughout we shall assume that $(X^\prime, \lambda^\prime)$ is close enough to $(X,\lambda)$ such that $\lambda^\prime \in V_0$.\\

Consider the lifts of  the crown ends $\mathcal{C}$ to the universal cover of $X$, together with the lifts of the arcs in  $\mathcal{L}_+$.  These are invariant under a Fuchsian group $\Gamma$; we choose a fundamental domain for this action on the combined set of crown ends and lifts of arcs. Namely, we get 

\begin{enumerate}
\item[(i)] a finite collection crown ends $\widetilde{\mathcal{C}}_1, \widetilde{\mathcal{C}}_1, \ldots, \widetilde{\mathcal{C}}_N$ and a  corresponding collection of fundamental domains $F_1,F_2,\ldots, F_N$ for the $\mathbb{Z}$-action on each of these crowns, and
\item[(ii)] a finite collection of arcs $\mathcal{A}_+$  (that are lifts of arcs of $\mathcal{L}_+$) between the ideal vertices determined by the boundary cusps of $F_1,F_2,\ldots, F_N$, 
\end{enumerate}
such that any other lift of an arc in $\mathcal{L}_+$ is taken to an arc in $\mathcal{A}_+$ by a unique element of $\Gamma$. 

Similarly, we have a finite collection of arcs $\mathcal{A}^\prime_+$ in the universal cover of $X^\prime$, that is the fundamental domain for the action of a Fuchsian group $\Gamma^\prime$ on the lifts of  $\mathcal{L}_+^\prime$.

Moreover, since $\mathcal{L}_+$ (resp. $\mathcal{L}_+^\prime  $ )  is a triangulation,  the collection $\mathcal{A}_+$ (resp. $\mathcal{A}^\prime_+$) is maximal, in the sense that we cannot add any other geodesic line to the collection that are between a pair of ideal vertices determined by  $F_1,F_2,\ldots F_N$ (resp $F_1^\prime ,F_2^\prime,\ldots F_N^\prime$) and are disjoint to the ones already in $\mathcal{A}_+$ (resp. $\mathcal{A}^\prime_+$). In particular, the arcs in $\mathcal{A}_+$ and $\mathcal{A}^\prime_+$ bound ideal triangles.\\

Now, as in the proof of Theorem \ref{nonuniq}, we consider a collection of ideal quadrilaterals $\mathcal{Q}$ determined by the geodesic lines in $\mathcal{A}_+$, namely, for each line in $\mathcal{A}_+$ the two adjacent ideal triangles determine an ideal quadrilateral $Q \in \mathcal{Q}$. Note that each pair of quadrilaterals in $\mathcal{Q}$ are either disjoint, or overlap along an ideal triangle.

Recall that by Theorem 1.1 of \cite{FG},  the decorated monodromy $\hat{\rho}$ uniquely determines the (complex) cross-ratios of quadruples $\{c_j,c_k,c_l,c_m\}$ of crown tips determined by the image of the lifts  of the crown ends by the developing map.  

In particular, the (complex) cross-ratio $\lambda_Q$ of an ideal quadrilateral in $Q\in \mathcal{Q}$ after grafting is determined by $\hat{\rho}$. Just as in the proof of Theorem \ref{nonuniq}, this uniquely specifies the real cross-ratio $\lvert \lambda_Q \rvert$ as well as the weight $w(Q)$ of the corresponding arc (which is a diagonal of $Q$). In particular, the weights on the arcs in $\mathcal{A}\subset \mathcal{A}_+$ and $\mathcal{A}^\prime\subset \mathcal{A}_+^\prime$, and hence $\mathcal{L}$ and $\mathcal{L}^\prime$, are uniquely determined and are equal. 
Thus $\mathcal{L} = \mathcal{L}^\prime$.

Moreover, the (real) cross-ratio  $\lvert \lambda_Q \rvert$ for each of the quadrilaterals $Q\in \mathcal{Q}$  uniquely determines the ideal vertices of  $F_1,F_2,\ldots, F_N$ as well as $F_1^\prime,F_2^\prime,\ldots, F_N^\prime$; this shows that the crown ends in $\mathcal{C}$ are isometric to those in $\mathcal{C}^\prime$.
\end{proof}

\medskip

To complete the proof, we need to show:

\begin{prop}\label{prop2} Suppose Equation \eqref{meq} holds, where the pair $(X^\prime, \lambda^\prime)$ is sufficiently close to $(X,\lambda)$ in   $\mathcal{T}_g(\mathfrak{n})\times   \mathcal{ML}_g(\mathfrak{n})$, 
then the hyperbolic surfaces-with-boundary $X_S$ and $X_S^\prime$ are isometric, and  $\lambda_S = \lambda_S^\prime$.
\end{prop}

\begin{proof} 
Let $\mathcal{L}_S$ and $\mathcal{L}_S^\prime$ be the geodesic arcs of $\mathcal{L} \cap X_S$ and $\mathcal{L}^\prime \cap X^\prime$, respectively. We already know from Proposition \ref{prop1} that the arcs in $\mathcal{L}_S$ and $\mathcal{L}_S^\prime$  have  identical weights and determine the same homotopy classes.

Let $\rho\in \bigchi_{g,k}$ be the representation in the usual $\pslc$- character variety of the punctured (or bordered) surface obtained by ``forgetting" the decorations at the punctures (\textit{c.f.} Definition \ref{decv}). Thus, $\rho$ is the image of $\hat{\rho}$ under the forgetful map
\begin{equation*}
p: {\widehat{\bigchi}}_{g,k}(\mathfrak{n}) \to \bigchi_{g,k}.
\end{equation*}

We then apply the Ehresmann-Thurston principle for manifolds with boundary (see, for example, Theorem I.1.7.1 of \cite{CEG} or Proposition 1 of \cite{Danciger}):  

Let $\mathcal{D}(S_{g,k}, \cp)$ be the space of developing maps for projective structures on the surface-with-boundary $S^0$ (homeomorphic to $S_{g,k}$) that are fixed on a collar neighborhood of the boundary.  This space is equipped with the usual compact-open topology. Then the Ehresmann-Thurston principle implies that  there is a neighborhood $W$ of the developing map $\bar{f}$  for the restriction of $P$ (that we fixed at the beginning of the section) to $S^0$, such that any developing map in $W$ that has the same holonomy $\rho$ as $P$, is equivariantly isotopic  to $\bar{f}$.  

In particular, in the space of projective structures on $S^0$, the restrictions of $P$ and $P^\prime = \widehat{Gr}(X^\prime,\lambda^\prime)$ are equivalent. Thus the Thurston construction in \S2.4 when applied to the corresponding developing maps, yields isometric hyperbolic surfaces $X_S$ and $X^\prime_S$, and identical grafting laminations $\lambda\cap X_S$ and $\lambda^\prime \cap X_S^\prime$. Since we already know the arcs in $\mathcal{L} \cap X_S$ and $\mathcal{L}^\prime \cap X^\prime$ have  identical weights and determine the same homotopy classes, we conclude that  $\lambda_S = \lambda_S^\prime$. 

This completes the proof.
\end{proof}

By Propositions \ref{prop1} and  \ref{prop2}, we conclude that if $(X^\prime,  \lambda^\prime)$ is sufficiently close to $(X,\lambda)$, and  Equation \eqref{meq} holds, then in fact  Equation \eqref{feq} is true, namely $X=X^\prime$ and $\lambda = \lambda^\prime$.
This shows that the monodromy map $\Phi$ is locally injective. Since we already know that  (a) $\Phi$ is continuous, (b)  the image lies in the smooth part of the decorated character variety, and (c) the dimensions at a point $P$ of  $\mathcal{P}_g(\mathfrak{n})$ and a smooth point of ${\widehat{\bigchi}}_{g,k}(\mathfrak{n})$ are identical (see Proposition \ref{dimchar}), we conclude that $\Phi$ is a local homeomorphism from the invariance of domain. 

This proves Theorem \ref{thm3}.

\appendix\section{Matching laminations}

The proof of Theorem \ref{mln} for parametrizing the space $\mathcal{ML}_g(\mathfrak{n})$ relies on the fact that we know parametrizations of the space of measured laminations on a surface-with-boundary, and a crown separately. In this appendix we give details of how we match two such laminations together to  obtain one on the entire crowned hyperbolic surface.

\subsubsection*{Splitting arcs}

We start with the following notion:

\begin{defn}[Properly homotopic arcs]\label{phe} Two arcs $\alpha_1, \alpha_2$ on a hyperbolic crown are said to be {\it properly homotopically equivalent} if both have one end-point on the boundary geodesic $\gamma$, both arcs end at the same crown-tip, and  both complete the same number of integer twists around $\gamma$. 

 Similarly, two arcs on a crowned hyperbolic surface are said to be {\it properly homotopically-equivalent} if 
 \begin{enumerate}
 \item their restrictions to the hyperbolic crown end are {\it properly homotopically-equivalent} in the sense above, and 
 \item their restrictions to the hyperbolic surface-with-boundary in the complement to the crown are homotopic arcs, where the homotopy is allowed to move endpoints on the boundary geodesic. 
 \end{enumerate}

 \end{defn} 

In what follows, a \textit{splitting} of a weighted isolated  arc  in a measured lamination shall refer to a replacement of the arc by several  properly homotopic  arcs (see Figure 9)  with weights that have the same sum. 

Note that for a crown end with more than one boundary cusp, this replacement can be done simultaneously for finitely many geodesic arcs  that cross the crown boundary and proceed to the boundary cusps.  Moreover, to have the correct \textit{marking} on the crown, we also need to maintain the (integer) number of twists of the arcs around the boundary component.

\begin{figure}
  \centering
  \includegraphics[scale=0.4]{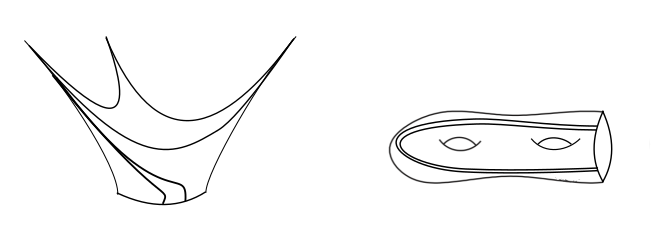}\\
  \caption{A pair of properly homotopic arcs on a crown (left) and a surface with boundary (right) obtained by splitting a single arc.}
\end{figure}

\subsubsection*{Determining the matching}

We shall denote the surface-with-boundary by ${S}$  the hyperbolic crown by $\mathcal{C}$, and the boundary of $S$ by $\gamma$. For simplicity of exposition, we assume here that $k=1$, that is, there is a single crown end; in the case $k>1$,  we can consider $\gamma$ to be a collection of closed geodesics, and $\mathcal{C}$ to a disjoint collection of hyperbolic crowns, and our argument holds for such a disconnected surface or boundary as well.

The identification of $S$ and $\mathcal{C}$ along $\gamma$  by a hyperbolic isometry  yields a hyperbolic crowned surface that we shall denote by $\hat{S}$.  Note that the boundary twist parameter (see \S2.4) is crucial to uniquely specify this identification.\\

Given measured laminations on $S$ and $\mathcal{C}$, such that the transverse measure of $\gamma$ induced by them are identical, we wish to construct a combined lamination on  $\hat{S}$.

The issue is that the leaves incident on the (common) boundary $\gamma$  might not match  --  indeed, the numbers of arcs incident on $\gamma$ from either side, or their endpoints, need not be same.

To resolve this, our strategy then would be to split these arcs on the subsurfaces and match the resulting arcs, such that the restriction of the resulting arcs on either subsurface still defines the same collection of homotopy classes of arcs. 
In this matching we also need to distribute the weights; for this, we shall need the following lemma. 

\begin{lem}\label{match} Let $R$ be a rectangle, with $n$ disjoint weighted arcs with weights $a_1,a_2,\ldots a_n$ (from left to right) incident on the top edge from outside $R$, and $m$ disjoint weighted arcs with weights $b_1,b_2,\ldots b_m$ (from left to right) incident on the bottom edge from outside $R$.  Suppose the total weights of the arcs incident on the top and bottom edges are the same, that is, $\sum\limits_{i=1}^n a_i = \sum\limits_{j=1}^m b_j$. 

 Then there is a unique way to split the arcs, and redistribute weights, such that 
 \begin{itemize}
 \item[(i)] the resulting arcs can be paired by  a collection of parallel arcs $\Gamma$ in $R$, and  paired arcs have equal weights, and
 \item[(ii)] no two arcs in $\Gamma$ connect to arcs arising from the same splitting, at both the top edge and bottom edge.
 \end{itemize} 
\end{lem}

\medskip 

\textit{Remarks.} 1.  A ``splitting" of an arc above refers to replacing an arc by finitely many disjoint copies that  then acquire a left-right ordering.  Any pair of such copies is then said to arise from the ``same" splitting. 

2. We shall call the final matching obtained in this Lemma a \textit{minimal matching} in light of property (ii) above, which ensures there are no unnecessary splittings. 

\medskip

\begin{proof}

Without loss of generality,  we shall assume $n\leq m$. The proof proceeds by induction on $n+m$.

Note that if $n=1$, then there is a unique arc $\alpha$  incident on the top edge. Indeed, then there is a matching:  split the arc $\alpha$  into exactly $m$ copies,  and assign weights $b_1,b_2,\ldots b_m$ to them (from left to right), and 
 connect each of the resulting arcs incident on the top edge, with the $m$ arcs incident on the bottom edge. 
It is easy to see that (i) and (ii) are satisfied. There is a unique such matching, because none of the $m$ arcs incident on the bottom edge can be split; the parallel arcs in $R$ continuing connecting to them would necessarily connect to arcs obtained by a splitting of $\alpha$ on the top edge, violating (ii). 

The inductive step is as follows: 

Consider the first arcs from the left incident on the top and bottom edges, denoted by $\alpha$ and $\beta$ respectively. Note that the weight of $\alpha$ is $a_1$ and the weight of $\beta$ is $b_1$.  There are three cases:

\textit{Case 1:}  If $a_1=b_1$, then no splitting of arcs is required: we connect the endpoints of $\alpha$ and $\beta$ by an arc in $R$. The number of unpaired arcs on the top and bottom edges each reduce by $1$.

\textit{Case 2:}  If $a_1>b_1$, then we split $\alpha$ into two arcs $\alpha_l$ and $\alpha_r$ and assign weights $b_1$  and $a_1-b_1$ to the left and right arcs, respectively.  We pair the left arc with $\beta$, and consider the remaining (unpaired) arcs.  Notice that there now $n$ unpaired arcs incident on the top edge, and $m-1$ unpaired arcs incident on the bottom edge. 

\textit{Case 3:}  If $a_1<b_1$, we split $\beta$  into two arcs of weights $a_1$ and $b_1-a_1$, and pair the left arc with $\alpha$. This time there are $n-1$ unpaired arcs incident on the top edge, and $m$ unpaired arcs incident on the bottom edge.

In all cases, the total number of unpaired arcs on the top and bottom edges have reduced by at least $1$, and the induction is complete. 

Note that  by construction, (i) is satisfied by this matching.  We now explain why property  (ii) also holds:
Recall in Case 2 we split the arc $\alpha$  and pair $\alpha_l$ with $\beta$, then even if $\beta$ had been created in a splitting in the previous step, there are no other arcs from that splitting to the right of $\beta$. Thus, in the next step, the unmatched arc $\alpha_r$ (or a splitting of it) is necessarily paired with an arc on the bottom edge that does not arise in the same splitting as $\beta$.  
A similar argument holds in Case 3, where $\beta$ is split. 

It is also easy to see that such a matching is unique: indeed, in any matching, there is a leftmost strand $\gamma$ through $R$ that connects an arc $\alpha$  incident on the top edge to an arc $\beta$ of the bottom edge. Note that $\alpha$ might have arisen from a splitting of one of the original arcs incident on the top edge, or $\beta$ might have arisen from a splitting, but not both, since otherwise (ii) would be violated. We can then conclude the first matching must have come from one of the 3 cases above. Repeating the same argument for the next strand, we see that the entire matching must have been obtained by the algorithm above. 
\end{proof}

\begin{figure}
  \centering
  \includegraphics[scale=0.3]{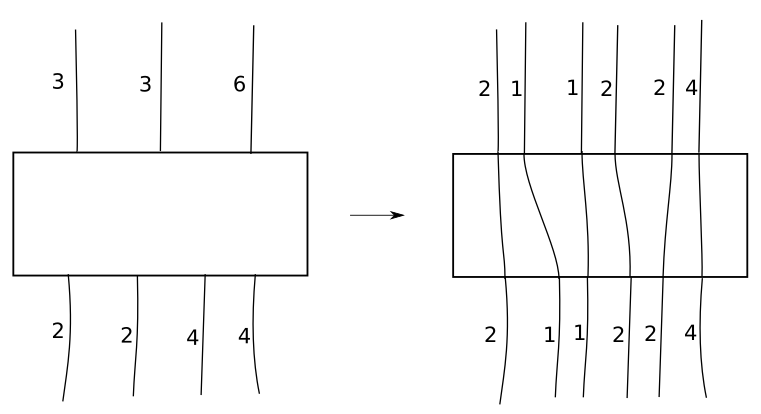}\\
  \caption{An example of a minimal matching (Lemma \ref{match}).}
\end{figure}

We now return to the situation where we need to match arcs from the measured lamination on the hyperbolic crown $\CC$, and surface-with-boundary $S$, that share a common geodesic boundary $\gamma$.

Let $\mathcal{G}= \{g_1, g_2,\ldots, g_m\}$ be the collection of weighted arcs on the hyperbolic crown $\mathcal{C}$ from the boundary cusps (labeled $1,2,\ldots, m$)  to the crown boundary.

Let the measured lamination on $S$ be $\lambda \cup \LL$, where $\lambda$ is a measured lamination that is disjoint from $\gamma$, and $\LL$ is a (non-empty) collection of weighted arcs incident on $\gamma$.\\

Recall from \S3.2 that the boundary twist parameter $\tau$ of the crown $\mathcal{C}$ can be thought of as a choice of a basepoint $p$ on the boundary curve $\gamma$, together with an integer twist parameter $t$ that records the number of topological Dehn-twists around $\gamma$ of arcs intersecting it (in our case, the arcs from $\mathcal{G}$ and $\LL$).
Recall that there is also ``canonical" basepoint $p_0$ for the crown. The remaining part of the twist parameter in fact determines the transverse measure of the arc of $\gamma$ between $p_0$ and $p$.

In the proof of the next Proposition, we shall use the point $p$ to cut up $\gamma$ into an interval, that is, it shall determine the fundamental domain for the action of the infinite cyclic group  corresponding to $\gamma$, on the universal cover of the crown. 

In case that $p$ coincides with an endpoint of one of the geodesic arcs in $\mathcal{G}$ or $\LL$, we split the corresponding arc such that the transverse measure of the arc on $\gamma$ between $p$ and $p_0$ remains the same, but the new arcs have endpoints distinct from $p$.

We shall assume that $p$ is distinct from the endpoints of the arcs in  $\LL$; else, we can change the arcs by a proper homotopy (as in Definition \ref{phe}) by sliding the endpoints along $\gamma$.  \\

See Definition \ref{phe} for the notion of ``properly homotopic" used below. 
We shall now prove:

\begin{prop}\label{strands} There is a unique way to split the weighted arcs in $\mathcal{G}$ and  $\LL$,  redistribute the weights and match the resulting arcs, such that:
\begin{itemize}
\item[(a)] arcs that are matched have the same weight, 
\item[(b)] a ``splitting" of an arc  $\alpha$ replaces it by properly homotopic copies (from boundary to the same boundary cusp in the case $\alpha \in \mathcal{G}$, and from boundary to boundary in case $\alpha \in \mathcal{L}$),
\item[(c)] after the redistribution of weights, the total weight of all the arcs arising from a splitting of an arc $\alpha$ (in $\mathcal{G}$ or $\LL$)  equals the original weight of $\alpha$. 
\end{itemize}

This results in a new collection of disjoint arcs $\hat{\LL}$ on the crowned surface $\hat{S}$, such that no two arcs of $\hat{\LL}$ are properly homotopic. 
Moreover, $\hat{\LL}$  together with $\lambda$, is a measured lamination on $\hat{S}$. 
\end{prop}

\begin{proof}
	
We describe the splitting and matching of the arcs in two stages. Lemma \ref{match} will be used several times. 

Let $\mathfrak{P} = \{p_1, p_2,  \ldots , p_N\}$ denote the endpoints of the arcs in $\LL$ on $\gamma$, where these points are ordered in the orientation of $\gamma$ induced from the orientation on the surface, and the basepoint $p$  on $\gamma$ lies between $p_N$ and $p_1$ (see the discussion above, preceding this Proposition). Note that each arc in $\LL$ actually determines two endpoints in the set $\mathfrak{P}$, and  this determines a pairing of the elements of $\mathfrak{P}$. \\

\textit{First stage:} 
Consider a thin closed annular neighborhood of $\gamma$, and cut along a geodesic arc perpendicular to $\gamma$ and passing through $p_0$, to obtain a rectangle $R$.
The arcs of $\mathcal{G}$ are incident on the top edge of $R$, and  half-arcs of $\mathcal{L}$ are incident on the bottom edge. (The half-arcs are paired to give the arcs in $\LL$, but we do not consider that fact in this first stage. The weights on the half-arcs are the same as that of the arc of $\LL$ they belong to.) 
The total weight of the arcs incident on the top and bottom edge of $R$ are the same by our assumption. Hence we can apply Lemma \ref{match}, which determines a unique minimal matching involving a splitting of the arcs of $\mathcal{G}$, and the half-arcs from $\mathcal{L}$ incident on the bottom edge of $R$. 

Let $\mathcal{G}^\prime $ be the set of arcs obtained by this initial splitting of the arcs in $\mathcal{G}$.  To each point  $p_i \in \mathfrak{P}$ we associate a subset  $\mathcal{G}_i^\prime\subset \mathcal{G}^\prime$ as follows:   if $l_i^+$ is the half-arc of an arc in $\LL$ incident on $\gamma$ at $p_i$, then  $\mathcal{G}_i^\prime$ comprises  all the arcs of $\mathcal{G}^\prime$ that are matched with a splitting of $l_i^+$.  

Note that property (ii) of the minimal matching (see Lemma \ref{match}) ensures that for any fixed $1\leq i\leq N$, the arcs of $\mathcal{G}_i^\prime$ are asymptotic to distinct cusps.\\

\begin{figure}
  \centering
  \includegraphics[scale=0.45]{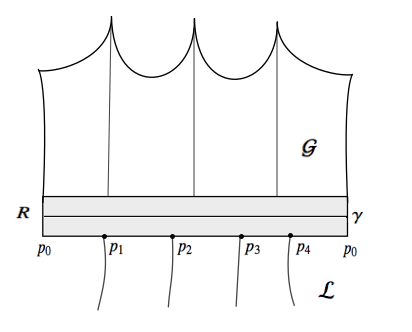}\\
  \caption{Lemma \ref{match} is used to determine a preliminary splitting of the arcs of $\mathcal{G}$ in the first stage of the construction.}
\end{figure}

\textit{Second  stage:} 
Now consider two points of $\mathfrak{P}$ that are paired, say $p_i$ and $p_j$.   That is to say, there is an arc  $l \in \LL$ , contained in the surface-with-boundary $S$ that has  endpoints $p_i$ and $p_j$ on $\gamma$. Then consider a rectangular neighborhood $R_l \subset S$ of the arc $l$ , where the top and bottom edges are segments of $\gamma$.  We can consider the arcs of $\mathcal{G}_i^\prime$ and $ \mathcal{G}_j^\prime$ as incident on these top and bottom edges. By property (i) of the minimal matching construction in the first stage, the total weight of the arcs in $\mathcal{G}_i^\prime$ is the same as that of the half-arc $l^+ \subset l$  that was incident on $p_i$,  that is, equals the weight of $l$. The same is true for the  the total weight of the arcs in $\mathcal{G}_j^\prime$, since that equals the weight of the other half-arc of $l$.
Hence, the total weights of the arcs incident on the top and bottom edges of $R_l$ are equal, and Lemma \ref{match} can be applied. 

The arcs in $R_l$ of the resulting minimal matching determines a splitting of the arc $l$, for each $l\in \LL$. 
Moreover, it determines a splitting of the arcs in $\mathcal{G}_i^\prime$  for each $i \in \{1,2,\ldots, N\}$ that completes the splitting of arcs in $\mathcal{G}$.\\

\begin{figure}
  \centering
  \includegraphics[scale=0.35]{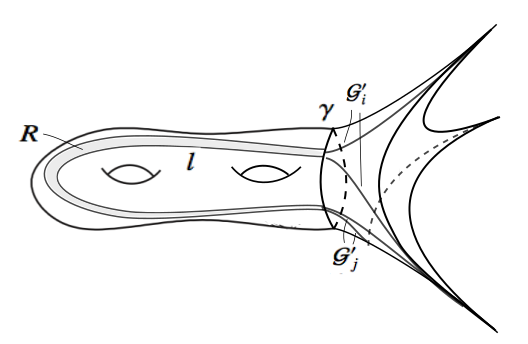}\\
  \caption{In the second stage, Lemma \ref{match} is used to determine a splitting of $l$, for each $l\in \LL$.}
\end{figure}

At the end of this second stage, we obtain a collection of weighted arcs $\hat{\LL}$ on the crowned surface $\hat{S} = S\cup_\gamma \CC$ between the boundary cusps of the crown.  It is easy to see that our construction ensures properties (a), (b) and (c) above. Moreover, we can verify that no two arcs, say $\hat{l}_1, \hat{l}_2$  in $\hat{\LL}$ are homotopic: indeed if they are,  their restriction to $S$ are homotopic, that is, they are splittings of the same $l \in \LL$.   However, recall that this splitting of $l$ is defined in  the second stage, where we determine a minimal matching of the arcs of  $\mathcal{G}_i^\prime$ and $ \mathcal{G}_j^\prime$ where $p_i$ and $p_j$ are the endpoints of $l$.  By property (ii) of a minimal matching, a pair of arcs obtained by splitting $l$ must connect to distinct splittings of arcs  in $\mathcal{G}_i^\prime$  or in $ \mathcal{G}_j^\prime$ (or both). However, we noted at the end of the first stage that distinct arcs in $\mathcal{G}_i^\prime$ or  $ \mathcal{G}_j^\prime$ are asymptotic to distinct cusps. Hence $\hat{l}_1$ and $\hat{l}_2$ are asymptotic to distinct cusps at one end (at least) which contradicts the assumption that they are homotopic. \\

We can now homotope each arc in  $\hat{\LL}$ to its geodesic representative,  and the fact that the homotopy classes are pairwise distinct ensures we obtain a set of weighted geodesic arcs on $\hat{S}$  of the same cardinality as $\hat{\LL}$. 
Together with the measured geodesic lamination $\lambda$ on $S$,  they determine a measured  geodesic lamination $\hat{\lambda}$ on $\hat{S}$. 

It only remains to show the uniqueness of such a measured lamination. This reduces to the uniqueness of the minimal matchings in the first and second stages, as follows:

Let $\hat{\lambda}_0$ be measured lamination on $\hat{S}$ that restricts to measured laminations on $S$ and $\mathcal{C}$ determined by the same data (i.e.\ the parameters described in the proof of Theorem \ref{mln}) as that of $\hat{\lambda}$.
Let $\hat{\LL}_0$ be the part of the measured lamination that is not compactly supported, comprising weighted arcs exiting the boundary cusps of the crown. Let ${\lambda}_0$ be the compactly supported part. 

The intersection of $\hat{\LL}_0$ with $S$ determines a collection of weighted arcs ${\LL}^\prime$ with endpoints on the boundary $\gamma$. Since the measured lamination ${\LL}^\prime \cup {\lambda}_0$ is, up to isotopy,  equal to ${\LL} \cup {\lambda}$,  we conclude that ${\lambda}_0 = \lambda$, and the arcs of $\LL^\prime$  must constitute a splitting of the arcs of $\LL$. Similarly, the  intersection of $\hat{\LL}_0$ with $\CC$ determines a collection of weighted arcs $\mathcal{G}^\prime_0$, that is a splitting of the arcs of $\mathcal{G}$.

Indeed, we shall now verify that $\hat{\LL}_0$ is obtained by the splitting-and-matching of the arcs $\mathcal{G}$ and $\mathcal{L}$ exactly as in the two-stage construction above. 

Consider the arcs $\hat{\LL}^\prime_l$ of $\hat{\LL}^\prime$ that correspond to a splitting of $l \in \LL$. Then the endpoints on $\gamma$ determine two collection of points $I_l^+$ and $I^-_l$, and we denote the corresponding collections of arcs of $\mathcal{G}^\prime_0$ incident on these point-sets by  $\mathcal{G}^\prime_{0,l,+}$ and  $\mathcal{G}^\prime_{0,l,-}$ respectively.  Here  $\mathcal{G}^\prime_{0,l,\pm}$ are splittings of a smaller  pair of arc-sets that we denote by $\mathcal{G}^\prime_{l,\pm}$, obtained by ``combining" arcs that are asymptotic to the same boundary cusp to a single arc (with a weight equal to the total weight of the combined arcs).  Note that this ensures that each $\mathcal{G}^\prime_{l,+}$ and $\mathcal{G}^\prime_{l,-}$ comprises arcs that are asymptotic to \textit{distinct} boundary cusps. 

 Since the original arcs $\hat{\LL}_0$ are pairwise homotopically distinct, a pair of arcs in the splitting of $l$ cannot connect to arcs of $\mathcal{G}^\prime_{0,l,+}$ and $\mathcal{G}^\prime_{0,l,-}$ that arise in the same splitting (of a pair of arcs in $\mathcal{G}$), at both of its ends. Thus, the collection of arcs $\hat{\LL}^\prime_l$ corresponding to a splitting of $l$ is a minimal matching of $\mathcal{G}_{l,\pm}^\prime$, exactly as in the second stage above, which is unique by Lemma \ref{match}. 
 
Finally we need to verify that the collections $\mathcal{G}^\prime_{l,\pm}$ as $l$ varies over $\LL$, are obtained by a minimal splitting exactly as in the first stage of the construction above. The arcs of $\mathcal{G}^\prime_{l,\pm} \subset \mathcal{G}^\prime_0$ are splittings of the arcs of $\mathcal{G}$, and the arcs connect to splittings of half-arcs of $\LL$. Hence these do constitute a matching, satisfying property (i) of Lemma \ref{match}. Property (ii) also holds, since by construction, the arcs of $\mathcal{G}^\prime_{l,\pm}$ are asymptotic to distinct cusps. Thus $\bigcup\limits_{l\in \LL}\mathcal{G}^\prime_{l,\pm}$ forms a minimal matching of $\mathcal{G}$ and half-arcs of $\LL$, which is unique by Lemma \ref{match}. 
This concludes the proof of uniqueness, and thus the proof of the Proposition.
 \end{proof}

\bibliographystyle{amsalpha}
\bibliography{mrefs}

\end{document}